\documentclass[11pt]{a_a_win}

\usepackage[T2A]{fontenc} 
\usepackage[cp1251]{inputenc} 
\usepackage[english,russian]{babel}
\usepackage{amsfonts,mathrsfs,amscd} 
\usepackage{amssymb,latexsym}
 \usepackage{amsmath} 
\usepackage{enumerate} 
\usepackage{verbatim}
\usepackage[mathscr]{eucal} 
\usepackage[dvips]{graphicx} 
\usepackage{marvosym}
\usepackage{textcomp}
\usepackage{wasysym}

\usepackage{xcolor} 

\usepackage[hyper]{amsbib}
\hypersetup{pdfstartview=FitH,linkcolor=blue,urlcolor=magenta,
citecolor=red,colorlinks=true}

\usepackage[hyper]{amsbib}
\usepackage{esint}

\theoremstyle{plain} 
\newtheorem{theorem}{Теорема}

\newtheorem{theoMR}{Теорема MR\hspace{-4pt}}
\newtheorem{corolMR}{Следствие MR\hspace{-4pt}} 
\newtheorem{theoKh}{Теорема Kh\hspace{-4pt}} 
\newtheorem{corolKh}{Следствие Kh\hspace{-4pt}}

\newtheorem{theoB}{Теорема B\hspace{-4pt}}

\newtheorem{lemma}{Лемма}[section]
\newtheorem{propos}{Предложение}[section] 
 
\theoremstyle{definition}
\newtheorem{definition}{Определение}[section] 
\newtheorem{remark}{Замечание}[section]

\theoremstyle{plain} 
\newtoks\thehProclaim 
\newtheorem*{Proclaim}{\the\thehProclaim}

\theoremstyle{definition} 
\newtoks{\thehRemark} \newtheorem*{Remark}{\the\thehRemark}

\renewcommand{\leq}{\leqslant} 
\renewcommand{\geq}{\geqslant}
\newcommand{\mc}{\mathcal} 
\newcommand{\rad}{\text{\tiny\rm rad}}
\newcommand{\bal}{\rm {bal}}
\newcommand{\Bal}{\rm {Bal}}
\newcommand{\RR}{\mathbb{R}} 
\newcommand{\CC}{\mathbb{C}} 
\newcommand{\NN}{\mathbb{N}} 
\newcommand{\ZZ}{\mathbb{Z}}
\newcommand{\DD}{\mathbb{D}} 
\newcommand{\e}{\varepsilon}
\newcommand{\const}{{\rm const}}

\DeclareMathOperator{\clos}{clos} 
\DeclareMathOperator{\Int}{int}
\DeclareMathOperator{\Har}{har} 
\DeclareMathOperator{\Hol}{Hol}

\DeclareMathOperator{\dens}{dens} 
\DeclareMathOperator{\Meas}{Meas} 
\DeclareMathOperator{\Exp}{Exp} 
\DeclareMathOperator{\Zero}{Zero} 
\DeclareMathOperator{\sbh}{sbh} 

\DeclareMathOperator{\dsbh}{\text{$\delta${\rm -sbh}}} 
\DeclareMathOperator{\supp}{supp} 

\DeclareMathOperator{\type}{type} 
\DeclareMathOperator{\ord}{ord}

\DeclareMathOperator{\up}{\text{\rm \tiny up}}
 \DeclareMathOperator{\lw}{\text{\rm \tiny low}}
\DeclareMathOperator{\rght}{\text{\rm \tiny right}}
 \DeclareMathOperator{\lft}{\text{\rm \tiny left}}
 
\DeclareMathOperator{\strip}{strip}

\DeclareMathOperator{\dd}{\,{\mathrm d\!}}
\DeclareMathOperator{\ind}{ind} 
\renewcommand{\Re}{{\rm Re \,}}
\renewcommand{\Im}{{\rm Im \,}}
\DeclareMathOperator{\Spf}{Spf} 
\DeclareMathOperator{\spf}{spf} 
 
\DeclareMathOperator{\width}{width}

\begin{document} 
\title[Выметание конечного рода на систему лучей]{Выметание мер и
субгармонических функций на систему лучей. III. 
Рост целых функций экспоненциального типа вдоль прямой}
	
\author{Б.\,Н.~Хабибуллин, А.\,Е.~Егорова} 
	
\address{факультет математики и ИТ\\ Башкирский государственный университет\\ 450074, г. Уфа\\ ул. Заки Валиди, 32\\ Башкортостан\\
Россия}
	
\email{Khabib-Bulat@mail.ru}
	
	
\begin{abstract} Во второй части настоящей работы была разработана 
техника выметания конечного рода $q=0,1,2,\dots$ меры (заряда) 
и ($\delta$-)су\-б\-г\-а\-р\-м\-о\-н\-и\-ч\-е\-с\-к\-ой функции конечного порядка на произвольную замкнутую систему лучей $S$ с вершиной в нуле на комплексной плоскости $\CC$. В настоящей третьей части работы мы используем только случай $q=1$, когда $S$ --- пара противоположно направленных лучей, т.\,е., $S$, как точечное множество, --- прямая, а выметание производится из обеих сторон этой прямой. При этом рассматриваются меры и субгармонические функции конечного типа при порядке $1$. Такое двустороннее выметание рода $1$ применяется к вопросам нетривиальности {\it весовых классов целых функций экспоненциального типа\/} $E$, выделяемых лишь ограничением на их рост вдоль прямой; к полному описанию подпоследовательностей нулей для таких классов $E$; к существованию целых функций-мультипликаторов $h$ для целых функций экспоненциального типа $g$, ограничивающих при умножении их рост классом $E$, т.\,е. $fh\in E$;   
к возможности представления мероморфной функции в виде частного функций из $E$. Истоки исследования --- в классической теореме Мальявена\,--\,Рубела об  условиях существования целой функции экспоненциального типа, обращающейся в нуль на заданной последовательности {\it положительных чисел.\/} Исследования эти также, в определенном смысле, параллельны знаменитым теоремам Бёрлинга\,--\,Мальявена о мультипликаторе и о радиусе полноты.

Библиография:  32 названия 
\end{abstract}
	
\thanks{Исследование выполнено за счёт гранта Российского
научного фонда (проект № 18-11-00002).}	
\date{}
	
\maketitle

\section{Введение}\label{s10}
\subsection{Цели}

Основная цель настоящей третьей и последующей четвёртой частей исследования  --- описать условия существования ненулевой  целой функции экспоненциального типа (пишем \textit{ц.ф.э.т.})
 $f\neq 0$ на {\it комплексной плоскости\/} $\CC$, обращающейся в нуль на заданной последовательности точек ${\sf Z}=\{{\sf z}_k\}\subset \CC$, с ограничениями на ее рост вдоль прямой. При этом будут существенно использованы  методы и техника первой \cite{KhI} и, в особенности, второй \cite{KhII} частей нашего исследования, касающиеся выметания рода $1$ мер и субгармонических функций на прямую. По традиции, восходящей к Л.\,А.~Рубелу и  П.~Мальявену, в качестве такой прямой, как правило,  выбираем {\it мнимую ось\/} $i\RR\subset \CC$, где $\RR\subset \CC$ --- {\it вещественная ось.\/} Такого рода задача была полностью решена в совместной работе Л.\,А.~Рубела и П.~Мальявена \cite[теорема 4.1]{MR} для \underline{{\it положительной\/}} последовательности точек ${\sf Z} \subset \RR^+:=\{x\in \RR\colon x\geq 0\}$, и с ограничением сверху $|f|\leq |g|$ на  $i\RR$, 
где  $g$ ---  заранее заданная ц.ф.э.т., обращающаяся в нуль на некоторой, вообще говоря, другой последовательности  ${\sf W}\subset \RR^+$. Этой задаче посвящена одна из основных частей монографии Л.\,А.Рубела в сотрудничестве с Дж.\,Э.~Коллиандром \cite[раздел 22]{RC} 1996 г. Но ещё в работах первого из авторов  1988--91 гг. все эти результаты были уже перенесены на произвольные 
\underline{{\it комплексные\/}}  последовательности ${\sf Z}, {\sf W} \subset \CC$
с ограничением сверху вида $\log |f|\leq p$ вдоль $i\RR$ через специальную 
 \textit{субгармоническую функцию-мажоранту\/}  $p$:
для $p(z)=\e |z|$ в \cite[основная теорема]{Kha88}  или с $p(z)=\log |g(z)|+ \e|z|$  со сколь угодно малым числом $\varepsilon >0$  \cite[теорема 1]{KhaD88}, \cite[основная теорема]{Kha89}. При  жестком требовании  $\e=0$  и {\it условии Картрайт\/}
вдоль $i\RR$ на функцию $g$, заключающемся в конечности интеграла $J\bigl(\log^+|g|\bigr)$, уже использованного в\footnote{По-прежнему ссылки над знаками бинарных отношений означают, что эти соотношения как-то связаны с приведёнными ссылками.}   \cite [4.6, (4.16)]{KhII}: 
\begin{equation}\label{fK:abp+}
J(v):=J_{-\frac{\pi}{2},\frac{\pi}{2}}^{[1]}(1, +\infty;v)
\overset{\text{ \cite[4.6,(4.16)]{KhII}}}{:=}
\int_1^{+\infty} \frac{v(-iy)+v(iy)}{y^2} \dd y,
\end{equation}
где $\log^+:=\max \{0,\log \}$, эта проблематика полностью вписывается в  знаменитые теоремы Бёрлинга\,--\,Мальявена о мультипликаторе и о радиусе полноты \cite[3.2]{Khsur}--\cite{MasNazHav05} и в определенном смысле окончательно  ими решается не только с ограничениями вдоль $i\RR$, но и с ограничениями на тип функции $f$ при порядке $1$. Этот простой, конечно же, только после теорем Бёрлинга\,--\,Мальявена,  случай будет в расширенном варианте  обсуждаться в намечаемых последующих статьях, продолжающих цикл работ, состоящих из настоящей статьи и тесно взаимосвязанных с ней статей  \cite{KhI}, \cite{KhII}.  
 
Наиболее сильные на начало  1990-х гг. результаты  для $\e=0$, но с функцией $g$, не удовлетворяющей условию Картрайт,   получены в \cite[основная теорема]{kh91AA} при условии расположения последовательностей  ${\sf Z}, {\sf W} \subset \CC$ вне какой-нибудь пары открытых углов, содержащих $i\RR\setminus \{0\}$. Аналогичные более общие результаты для случая  субгармонической функции-мажоранты  $p$ конечного типа при порядке $1$, гармонической вне такой же пары углов, затрагивались в диссертации  \cite[теорема 2.4.1]{KhDD92}, но в научных журналах они не публиковались. 
В данной третьей части работы мы их существенно развиваем на основе выметания рода $1$ на мнимую ось $i\RR$ мер и субгармонических функций.

Частичная  сводка этих результатов дана в \cite[3.2]{Khsur} и ниже в подразделе \ref{bal_iR}.  При этом возникла необходимость  обобщить,  развить и распространить введенные Л.\,А.~Рубелом и П.~Мальявеном ранее только {\it для положительных последовательностей\/} $\sf Z\subset \RR^+$ различные виды плотности их распределения около бесконечности  {\it на комплексные последовательности\/} $\sf Z\subset \CC$ и на меры на $\CC$. В настоящей третьей части работы и намечаемой четвёртой части мы существенно развиваем эти результаты и охватываем ряд новых вопросов. Определенные параллели  в этих наших исследованиях прослеживаются со знаменитыми теоремами Бёрлинга\,--\,Мальявена о мультипликаторе и о радиусе полноты \cite[3.2]{Khsur}--\cite{MasNazHav05}.
   
Из наших результатов, развивающих теорему Рубела\,--\,Мальявена, будут выведены в четвертой части настоящего цикла работ  дополнительные  критерии (не)полноты экспоненциальных систем функций в (под)про\-с\-т\-р\-а\-н\-с\-т\-в\-ах голоморфных функций в неограниченных областях $D\subset \CC$  в различных топологиях, а также в пространствах функций на {\it вещественной оси\/} $\RR$. Наряду с результатами различных авторов в этом направлении, достаточно полно освещенных в  монографии-обзоре первого из авторов \cite[3.2]{Khsur}, из  свежих результатов 2017 г., не вошедших в наш обзор, отметим совместную работу А.\,С.~Кривошеева  и А.\,Ф.~Кужаева \cite[теорема 4]{KriKuz17}, где в форме критерия  даны условия полноты экспоненциальной системы с {\it положительными показателями, вкладываемыми  в измеримые последовательности,\/} т.\,е.  очень  <<правильно распределенными>>,  для пространств голоморфных функций в {\it выпуклой области\/} $D$, включающей или нет в себя  отрезок заданной длины, параллельный мнимой оси. 

\subsection{Основные обозначения, определения и соглашения}\label{s1.1} Первоначально этот подраздел~\ref{s1.1}, по-видимому, целесообразно пропустить и обращаться к нему лишь  по мере необходимости. Подавляющая часть обозначений и  определений согласована с первой \cite{KhI} и второй \cite{KhII} частями нашей работы и в значительной мере соответствует таковым  из  монографии-обзора \cite{Khsur}  первого из авторов. Здесь мы напоминаем те из них, которые непосредственно потребуются в настоящей третьей части работы, дополняя некоторыми новыми определениями и обозначениями.

\subsubsection{Множества, порядок, топология}\label{1_1_1} Как обычно, $\NN=\{1,2,\dots\}$ и $\ZZ$ --- множества {\it натуральных\/} и {\it целых чисел,\/} $\NN_0:=\{0\}\cup \NN$ --- <<французский>> натуральный ряд.  
 {\it Расширения вещественной прямой\/} $\RR$
\begin{equation}\label{R+-}
\RR_{-\infty}:=\{-\infty\}\cup \RR, \quad \RR_{+\infty}:=\RR\cup \{+\infty\},\quad  \RR_{\pm\infty}:=\RR_{-\infty}\cup\RR_{+\infty} 
\end{equation}
рассматриваем  с естественными отношением порядка $-\infty \leq x\leq +\infty$ для любого $x\in \RR_{\pm\infty}$.  Для $\varnothing  \subset \RR_{\pm\infty}$ по определению $\sup \varnothing:=-\infty$ и 
$\inf \varnothing:=+\infty$.  \textit{Расширенная числовая прямая\/} $\RR_{\pm\infty}$  снабжается естественной порядковой топологией: \textit{базу открытых множеств\/} образуют множества $(a,b):=\{x\in \RR_{\pm \infty} \colon 
a<x<b \}$, $(a,+\infty]:=\{x\in \RR_{\pm \infty} \colon 
a<x \}$,  $[-\infty,b):=\{x\in \RR_{\pm \infty} \colon 
x<b \}$ с произвольными $\RR \ni a<b\in \RR$.  \textit{Интервал} --- связное подмножество в $\RR_{\pm \infty}$. Для $x\in \RR_{\pm \infty}$ и  $X\subset \RR_{\pm \infty}$ полагаем 
\begin{equation}\label{+}
x^+:=\max\{0,x \}, \;X^+:=\{x^+\colon  x\in X\}, \;X_*:=X\setminus \{0 \},\; X^+_*:=(X_*)^+. 
\end{equation}

$\CC_{\infty}:=\CC\cup\{\infty\}$ --- {\it расширенная комплексная плоскость,\/} или одноточечная компактификация  Александрова комплексной плоскости $\CC$ с естественной нормой-модулем $|\cdot|$; $|\infty|:=+\infty$. $S_*:=S\setminus \{0\}$ при $S\subset \CC_{\infty}$. 
 
$D(z,r):=\{z' \in \CC \colon |z'-z|<r\}$ --- {\it открытый круг с центром\/ $z\in \CC$ радиуса\/ $r\in \RR_{+\infty}^+$}; 
 $D(r):=D(0,r)$; $\DD:=D(1)$ --- {\it
единичный круг.\/} Открытые круги с центром  $z=\infty$  в настоящей статье удобно определить как
\begin{equation}\label{Dinfty}
D(\infty,r):=\{z\in \CC_{\infty} \colon |z|>1/r\}, \quad  \frac{1}{+\infty}:=0, \quad
\frac{1}{0}:=+\infty.
\end{equation} 
Тогда $\bigl\{D(z,r)\colon  z\in \CC_\infty, r\in \RR_*^+\bigr\}$ --- б\textit{аза открытых множеств в\/} $\CC_\infty$. 
Кроме того, $\overline{D}(z,r):=\{z' \in \CC \colon |z'-z|\leq r\}$ --- {\it замкнутый круг;} $D_*(z,r):=\bigl(D(z,r)\bigr)_*$, $D_*(r):=D_*(0,r)$, $\overline{D}_*(z,r):=\bigl(\overline{D} (z,r)\bigr)_*$, $\overline{D}_*(r):=\overline{D}_ *(0,r)$. 
\textit{Область} --- открытое связное подмножество в $\CC_{\infty}$.
Для $S\subset \CC_{\infty}$ полагаем $S_{\infty}:=S\cup \{\infty\}$. 
Так, $\RR_{\infty}=\RR\cup \{\infty\}\subset \CC_{\infty}$ --- {\it замыкание $\RR$ в $\CC_{\infty}$}.

\begin{equation}\label{C+-}
\begin{split}
\CC^{\up}:= \{z\in \CC \colon \Im z>0\},\quad &\CC_{\lw}:=\{z\in \CC\colon \Im z<0\}, \\
\CC_{\rght}:= \{z\in \CC \colon \Re z>0\},\quad 
&{_{\lft\!}\CC}:=\{z\in \CC\colon \Re z<0\},\\
\CC^{\overline \up}:=\CC^{\up}\cup \RR, 
\quad &\CC_{\overline \lw}:=\CC_{\lw}\cup \RR,\\
\quad \CC_{\overline \rght}:=\CC_{\rght}\cup i\RR,
\quad &_{\overline \lft\!}\CC:={_{\lft\!}\CC}\cup i\RR
\end{split}
\end{equation} 
--- соотв.\footnote{Сокращение для <<соответственно>>} {\it открытые верхняя, нижняя, правая и левая полуплоскости,\/} а  
также их {\it замыкания  в топологии комплексной плоскости\/} $\CC$.

На  $S\subset \CC_{\infty}$ индуцируется топология с $\CC_{\infty}$, но при рассмотрении
$S\subset \RR_{\pm\infty}$ как подмножества в $\RR_{\pm\infty}$ уже с $\RR_{\pm\infty}$. Для $z\in \CC$ и $S\subset \CC$ полагаем $zS:=\{zz'\colon z'\in S \}$, что вполне согласуется с обозначениями $i\RR$ и $i\RR^+$  соотв.  для {\it мнимой оси\/} и  {\it мнимой положительная полуоси\/}. 
Для $z\in \CC$ и $S\subset \CC_{\infty}$,  $\bar z$, как обычно, сопряжённое к $z$ число, $\bar \infty:=\infty$, $\bar S:=\{\bar z\colon z\in S\}$. 
Одним и тем же символом $0$ обозначаем, по контексту, число нуль, нулевой вектор, нулевую функцию, нулевую меру (заряд) и т.\,п.; $\varnothing$ --- {\it пустое множество.\/} Для подмножества $X$ векторной решётки  функций, мер и т.\,п. с отношением порядка $\geq$ также используем обозначения \eqref{+}.
{\it Положительность\/} понимается  как $\geq 0$. Аналогично \textit{отрицательность} --- $\leq 0$.   

Для подмножества $S$ топологического пространства $T$ через $\clos S$, $\Int S$, $\partial S$ обозначаем соотв. {\it замыкание,\/} {\it внутренность\/} и {\it границу\/} множества $S$ в $T$. Если замыкание $\clos S$ --- компакт в $T$, то $S$ {\it предкомпактно в $T$} и пишем  $S\Subset T$. В частности, для $S\subset \CC_{\infty}$ через $\clos S$, $\Int S$ и $\partial S$ обозначаем соотв. {\it замыкание,\/}
{\it внутренность,\/} и {\it границу\/} подмножества $S$ {\it в\/} $\CC_{\infty}$.

\subsubsection{Функции}\label{1_1_2}  
Произвольной функции $f\colon X\to \RR_{\pm \infty}$ сопоставляем 
пару функций $f^+\colon x\mapsto (f(x))^+$, $f^-:=(-f)^+$.
Функция $f\colon X\to Y$ с упорядоченными $(X,\leq)$ и
$(Y,\leq)$ {\it возрастающая}, если для $x_1,x_2\in X$ из $x_1\leq x_2$ следует $f(x_1)\leq f(x_2)$. Аналогично для убывания. В основном  $Y\subset \RR_{\pm \infty}$. {\it На множествах функций\/} с упорядоченным
множеством значений {\it отношение порядка\/} индуцируется с множества значений как {\it поточечное.} 

Для  $X\ni x_0$ и  $f\colon X\setminus \{x_0\}\to \RR_{\pm\infty}$ записи <<$f(x)\leq O(1)$ {\it  при\/} $x\to x_0$>> или $f(x)\underset{x\to x_0}{\leq} O(1)$  означает, что функция $f$ в какой-то  проколотой окрестности точки $x_0$ ограничена сверху некоторым числом из $\RR$. 

Пусть $S\subset \CC_{\infty}$. Классы $\Hol(S)$ и  $\sbh (S)$ состоят из сужений на $S$ функций, соотв. {\it голоморфных\/} и {\it субгармонических\/}  в каком-либо открытом множестве,  включающем в себя $S$ \cite{Rans}, \cite{HK};  
$\Har (S):=\sbh (S)\cap \bigl(-\sbh(S)\bigr)$ --- класс {\it гармонических\/} функций на $S$; $\dsbh (S):=\sbh (S) -\sbh(S)$ --- класс \textit{$\delta$-субгармонических,\/} или разностей субгармонических,  \textit{функций\/} на $S$ \cite{Ar_d}, \cite[2]{Gr}, \cite[3.1]{KhaRoz18}.
Полагаем  $\Hol_*(S):=\Hol(S)\setminus \{0\}$.  Тождественную $-\infty$ или $+\infty$ на $S$ обозначаем соотв. $\boldsymbol{-\infty}\in \sbh (S)$
или $\boldsymbol{+\infty}\in -\sbh (S)\subset \dsbh (S)$; $\sbh_*(S):=\sbh (S)\setminus \{\boldsymbol{-\infty}\}$, $\dsbh_*(S):=\dsbh (S)\setminus \{\boldsymbol{\pm\infty}\}$. В то же время для топологии на $S$, индуцированной с $\CC_{\infty}$,
  $C(S)$ --- банахово пространство над $\CC$  непрерывных функций $f\colon S\to \CC$  с $\sup$-нормой  $\|f\|_S:=\sup_{z\in S}|f(z)|$. 

Для {\it области\/} $D$ пространство $\Hol(D)$ над полем $\CC$ снабжается топологией равномерной сходимости на компактах, если не оговорено противное.

Для \textit{открытого подмножества\/} $O\subset \CC_{\infty}$ 
через $C(\clos O)\cap \Hol(O)$ обозначаем нормированное пространство функций $f\in C(O)$ с сужением $f\bigm|_O\in \Hol(O)$ с $\sup$-нормой $\|\cdot \|_{\clos O}$.     

Для {\it замкнутого подмножества\/}  $S\subset \CC_{\infty}$
{\it пространство ростков\/} $\Hol (S)$ состоит из классов эквивалентности при факторизации относительно отношения эквивалентности
\textit{<<функции равны на некотором открытом множестве $O\supset S$>>}. 
Пусть $S$---  пересечение открытых множеств $O_m\subset \CC_\infty$.
\textit{Топология индуктивного предела} на $\Hol(S)$ 
определяется базой окрестностей нуля,  получаемых как  
 абсолютно выпуклые оболочки всевозможных объединений непустых открытых шаров ненулевого  радиуса с центрами в нуле из нормированных пространств  $C(\clos O_m)\cap \Hol(O_m)$  \cite[0.1.2]{Khsur}.

Для $z\in \CC$ и чи\-с\-ла $r \geq 0$ определим  интегральные средние по  окружности $\partial \overline D(z, r)$  от 
 фу\-н\-к\-ц\-ии $ v\colon \partial 
  D(z,r)\to \RR_{\pm\infty}$: 
\begin{subequations}\label{df:MCB}
\begin{align}
C_v(z, r)&:=:C(z, r;v)\overset{\eqref{LS}}{:=}\frac{1}{2\pi} \int_{0}^{2\pi}  v(z+re^{i\theta}) \dd \theta, \; C_v(r):=C_v(0, r),\tag{\ref{df:MCB}c}\label{df:MCBc}\\
\intertext{по кругу $D(z,r)$ от фукции $v\colon D(z,r)\to \RR_{\pm \infty}$:}
B_v(z,r)&:=:B(z,r;v)\overset{\eqref{LS}}{:=}\frac{2}{r^2}\int_{0}^{r} C_v(z, t)t\dd t ,
\quad B_v(r):=B_v(0, r), \tag{\ref{df:MCB}b}\label{df:MCBb}\\
\intertext{а также верхнюю грань функции $v\colon \partial D(z,r)\to \RR_{\pm \infty}$ на  $\partial D(z,r)$:}
M_v(z,r)&:=:M(z,r;v):=\sup_{z'\in \partial D(z,r)}v(z'), 
\quad M_v(r):=M_v(0, r),  
\tag{\ref{df:MCB}m}\label{df:MCBm}
\end{align}
\end{subequations}
что  при $v\in \sbh \bigl( \overline D(z, r)\bigr)$ совпадает с верхней гранью 
функции $v$ в круге $\overline D(z, r)$. Конечно же, в  \eqref{df:MCB} для  \eqref{df:MCBc} и \eqref{df:MCBb} подразумевается существование интегралов  \cite[определение 2.6.7]{Rans}, \cite{HK}, для которых, вообще говоря, допускаются значения $\pm\infty$. 
Функция $v\in \sbh(\CC)$ {\it конечного типа\/} (при порядке $1$), если \cite[2.1]{KhI}
\begin{equation}\label{typev}
\type_1[v]\overset{\text{\cite[2.1]{KhI}}}{:=}\type_1^{\infty}[v]\overset{\eqref{df:MCBm}}{:=}
\limsup_{r\to +\infty}\frac{M_{v}(r)}{r}<+\infty.
\end{equation}
Функция $v\in \dsbh_*(\CC) $ конечного типа (при порядке $1$), если она представима\footnote{Исходное определение естественнее давать как функцию $v$ с   характеристикой Неванлинны $T_v$ конечного типа (при порядке $1$), но для функций конечного порядка это классическое определение эквивалентно приводимому здесь.} в виде  разности двух функций из $\sbh_*(\CC)$ конечного типа.
Функция $f\in \Hol(\CC)$ ---
{\it целая функция\/  экспоненциального типа (ц.ф.э.т.),\/} если 
 функция $\log |f|\in \sbh(\CC)$ конечного типа (при порядке $1$).

Произвольной функции $u\colon \CC\overset{\eqref{R+-}}{\to} \RR_{\pm \infty}$ сопоставляем ее  {\it индикатор роста\/} при порядке $p\in \RR^+$, а именно:
\begin{equation}\label{h}
{\ind_p}[u]\colon \theta \mapsto\limsup_{\RR^+\ni r\to +\infty}\frac{u(re^{i\theta})}{r^p}, \quad \theta \in \RR,
\end{equation}
--- $2\pi$-периодическая функция. Если $v\in \sbh_*(\CC)$ --- субгармоническая фу\-н\-к\-ц\-ия конечного типа  и  $f\in \Hol_*(\CC)$ --- ц.ф.э.т., то ${\ind_1}[v]\colon \RR\to \RR_{-\infty}$, а также  ${\ind_1}[\log|f|]\colon \RR\to\RR_{-\infty}$ ---  \textit{тригонометрически выпуклые функции\/} \cite{Levin56}, \cite{Levin96}. 

Число $C$ и постоянную функцию, тождественно равную $C$, не различаем. Через $\const_{a_1, a_2, \dots}$ обозначаем вещественные постоянные, зависящие от $a_1, a_2, \dots$ и, если не
оговорено противное, только от них.

\subsubsection{Меры и заряды}\label{1_1_3} Далее $\Meas (S)$ --- класс\footnote{В \cite{KhI}, \cite{KhII}  вместо $\Meas$ использоваловался символ  $\mathcal M$.} всех {\it счетно-аддитивных функций борелевских
подмножеств борелевского множества $S\subset \CC_{\infty}$ со значениями в $\RR_{\pm\infty}$, конечных на компактах из $S$.\/}
Элементы из $\Meas (S)$ называем {\it зарядами,\/} или {\it вещественными мерами,\/} на $S\subset \CC_{\infty}$; $\Meas^+ (S):=\bigl(\Meas (S)\bigr)^+$ --- подкласс {\it положительных мер,\/} или просто \textit{мер.\/} Заряд $\mu \in \Meas (S)$ {\it сосредоточен\/} на
$\mu$-измеримом подмножестве $S_0\subset S$, если $\mu (S')=\mu (S'\cap S_0)$ для любого $\mu$-измеримого подмножества $S'\subset S$.
Очевидно, заряд $\mu$ сосредоточен на \textit{носителе\/} $\supp \mu\subset S$. Для измеримого по $\mu\in \Meas(S)$ подмножества $S_0\subset
S$ через $\mu\bigm|_{S_0}$ обозначаем {\it сужение заряда $\mu$ на $S_0$.\/} Через $\lambda_{\CC}$, $\lambda_{\RR}$ и $\lambda_{i\RR}$ обозначаем {\it лебеговы меры\/} соотв. на $\CC$, $\RR$ и $i\RR$,  а также их сужения на подмножества из $\CC$. Нижний индекс $\CC$ в $\lambda_{\CC}$ при
этом часто опускаем. Через $\delta_z$ обозначаем {\it меру Дирака в точке $z\in \CC_{\infty}$,\/} т.\,е. вероятностную меру с $\supp \delta_z=\{z\}$. Как обычно, для $\mu\in \Meas(\CC)$ через
$\mu^+:=\max\{0,\mu\}$, $\mu^-:=(-\mu)^+$ и $|\mu|:=\mu^++\mu^-$, обозначаем {\it верхнюю, нижнюю\/ {\rm и} полную вариации заряда\/
$\mu=\mu^+ -\mu^-$.} Для заряда $\nu\in \Meas (S )$ и круга $D (z,r)\subset S$ полагаем
\begin{equation}\label{df:nup} 
\nu (z,r):=\nu \bigl(\,\overline D(z,r)\bigr),\quad \nu^{\rad}(r):=\nu(0,r) 
\end{equation} --- соотв.  {\it считающая функция с
центром\/ $z$} и {\it считающая радиальная фу\-н\-к\-ц\-ия} заряда $\nu$. 
Заряд $\nu\in \Meas (\CC)$ \textit{конечной верхней плотности\/} (при порядке $1$), если \cite[2.1]{KhI}
\begin{equation}\label{fden}
{\type_1}[\nu]\overset{\text{\cite[2.1]{KhI}}}{:=}\type_1^\infty [\nu]:=\limsup_{r\to +\infty} \frac{|\nu|^{\rad}(r)}{r} <+\infty.
\end{equation}

Для $\nu\in \Meas(\CC)$ используем
и {\it функцию распределения\/} $\nu^{\RR}$ сужения $\nu\bigm|_{\RR}$ заряда $\nu$ {\it на\/} $\RR$ \cite[(1.9)]{KhI}:
\begin{equation}\label{nuR} 
\nu^{\RR}(x):=\begin{cases} -\nu\bigl( [x,0) \bigr) \; &\text{ при } x<0 ,\\ \nu\bigl([0,x] \bigr) \; &\text{ при } x\geq 0,
 \end{cases} 
\end{equation}
а также {\it функцию распределения\/} $\nu^{i\RR}$ сужения $\nu\bigm|_{iR}$ заряда $\nu$ {\it на\/} $\RR$:
\begin{equation}\label{nuiR} 
\nu^{i\RR}(y):=\begin{cases} -\nu\bigl( [iy,0) \bigr) \; &\text{ при } y<0 ,\\ \nu\bigl([0,iy] \bigr) \; &\text{ при } y\geq 0,
 \end{cases} 
\end{equation}
которые являются функциями локально ограниченной вариации. Интегралы Лебега\,--\,Стилтьеса по интервалу по таким функциям $n\colon \RR\to \RR$ обычно, если не оговорено противное, понимаем как 
 \begin{equation}\label{LS}
\int_r^R\dots \dd n :=\int_{(r,R]}\dots \dd n 
\end{equation}
Для открытого $\mc O\subset \CC_{\infty}$ {\it меру Рисса\/} функции $v\in \sbh_* ({\mathcal O})$ обозначаем как $\nu_v:=\frac{1}{2\pi}\Delta v\in \Meas ^+(\mathcal O)$ или
$\mu_v$ и т.\,п., где оператор Лапласа $\Delta$ действует в смысле теории обобщённых функций. 
Для функции  $\boldsymbol{-\infty}\in \sbh (\mathcal O)$ её мера Рисса по
определению равна $+\infty$ на любом подмножестве из ${\mathcal O}$.
Функции $v\in \dsbh (\mathcal O)$ сопоставляется 
{\it заряд Рисса\/} $\nu_v:=\frac{1}{2\pi} \Delta v\in \Meas(\mathcal O)$.   
Для функции  $\boldsymbol{+\infty}\in -\sbh (\mathcal O)\subset \dsbh(\mathcal O)$ её заряд Рисса по определению равен $-\infty$ на любом подмножестве из ${\mathcal O}$.
\subsubsection{Последовательности точек в области} Пусть ${\sf Z}=\{ {\sf z}_k\}$ --- не более чем счётное  множество точек области $D\subset {\CC}$, проиндексированное некоторым набором индексов $k\in \ZZ$. В ${\sf Z}$ допускаются совпадения нескольких  точек с различными индексами. Однако всегда предполагаем, что  ${\sf Z}$  не имеет предельных  точек в $D$, если не оговорено противное. Возможно, что ${\sf Z}$ конечно или ${\sf Z} =\varnothing$. Такое {\it проиндексированное множество\/} ${\sf Z}$
естественнее воспринимать как целочисленную положительную функцию, т.\,е.
как {\it положительный дивизор\/} ${\sf Z} \colon D \to \ZZ^+$, равный  в каждой точке $z \in D$ числу вхождений точки $z$ в ${\sf Z}$, или как целочисленную {\it считающую меру}
\begin{equation}\label{df:divmn}
n_{\sf Z}:=\sum_k \delta_{{\sf z}_k}, \quad  
n_{\sf Z} (S)= \sum_{z \in S} {\sf Z} (z)=\sum_{{\sf z}_k\in S} 1 
 \quad \text{для любых $S\subset D$}.
\end{equation}
В частности, $n_{\sf Z} \bigl(\{ z\}\bigr)={\sf Z} (z)$ при всех $z \in D$.
Такой подход отличается от стандартного  взгляда на последовательность точек как на функцию натурального или целого аргумента. При последнем стандартном подходе, когда нумерация последовательности имеет значение, изображаем ее в круглых скобках, т.\,е. в виде ${\sf Z}=({\sf z}_k)$.  
Тем не менее как {\it проиндексированное множество\/} ${\sf Z}=\{{\sf z}_k\}$, так и {\it занумерованную последовательность\/} ${\sf Z}=({\sf z}_k)$ часто называем просто {\it последовательностью точек.\/} 

{\it Две последовательности\/} ${\sf Z}=\{{\sf z}_k\}$ и ${\sf W}
=\{ {\sf w}_k\}$ из $D$ {\it  равны\/} и соотв. пишем ${\sf Z} = {\sf W}$, если для  их дивизоров имеет место тождество ${\sf Z} (z )\equiv {\sf W} (z)$ при всех $z \in D$, т.\,е. $n_{\sf Z}\overset{\eqref{df:divmn}}{=} n_{\sf W}$ как меры. 
Иначе говоря, каждая {\it последовательность точек\/} рассматривается как представитель некоторого класса эквивалентности, состоящего  из последовательностей в $D$ с одинаковыми дивизорами.
При этом носитель $\supp {\sf Z}:=\supp n_{\sf Z}$. Запись $z \in {\sf Z}$ (соотв. $z \notin {\sf Z}$) означает, что $z \in \supp {\sf Z}$ (соотв. $z \notin \supp {\sf Z}$). Для подмножества $S \subset {\CC}$ запись ${\sf Z} \subset S$ означает, что 
$\supp {\sf Z} \subset S$; ${\sf Z} \cap S$ --- сужение последовательности  $\sf Z$ на $S$ с дивизором ${\sf Z} \bigm|_S$ и считающей мерой $n_{\sf Z}\bigm|_S$.

Последовательность точек ${\sf W} \subset D$ включена (содержится) в ${\sf Z}$, если $ {\sf W} (z)\leq {\sf Z}(z)$ в терминах дивизоров  при всех $z \in D$, т.\,е. $n_{\sf W}\overset{\eqref{df:divmn}}{\leq} n_{\sf Z}$. При этом пишем ${\sf W}\subset {\sf Z}$ и говорим, что ${\sf W}$
--- {\it подпоследовательность\/} из ${\sf Z}$.

{\it Объединение\/}  ${\sf Z} \cup {\sf W}$
через дивизоры задается тождеством 
$({\sf Z} \cup {\sf W})(z)\equiv {\sf Z} (z) + {\sf W} (z)$, т.\,е.
$n_{{\sf Z} \cup {\sf W}}\overset{\eqref{df:divmn}}{=}n_{\sf Z}+n_{\sf W}$, а \textit{пересечение\/} ${\sf Z} \cap {\sf W}$ через дивизоры --- тождеством  
$({\sf Z} \cup {\sf W})(z)\equiv \min \bigl\{{\sf Z} (z) , {\sf W} (z)\bigr\}$, т.\,е.
$n_{{\sf Z} \cap {\sf W}}\overset{\eqref{df:divmn}}{=}\inf \{n_{\sf Z},n_{\sf W}\}$. При ${\sf W}\subset {\sf Z}$ {\it разность\/} последовательностей  ${\sf Z} \setminus {\sf W}$ определяет дивизор $({\sf Z} \setminus {\sf W})(z) \equiv {\sf Z}(z) -  {\sf W} (z)$, $z \in D$, или считающая мера $n_{{\sf Z} \setminus  {\sf W}}\overset{\eqref{df:divmn}}{=}
n_{\sf Z}-n_{\sf W}$. На последовательностях  точек операции и отношения, отличные от приведенных выше,  понимаются поэлементно. Так,
$z{\sf Z} :=\{ z{\sf z}_k\}$, $\Re {\sf Z} := \{ \Re {\sf z}_k\}$;
${\sf Z} \geq 0$, если ${\sf z}_k \geq 0$ для всех $k$, и т.\,п.

В соответствии с \eqref{df:nup}--\eqref{nuiR} определяются  \begin{equation}\label{nZ}
\begin{split}
n_{\sf Z}(z, r), \quad n_{\sf Z}^{\rad}(r)\overset{\eqref{df:nup}}{:=}&(n_{\sf Z})^{\rad}(r), \quad r\in \RR^+,
\\
n_{\sf Z}^{\RR}(x)\overset{\eqref{nuR}}{:=}(n_{\sf Z})^{\RR}(x), 
\quad n_{\sf Z}^{i\RR}(y)\overset{\eqref{nuiR}}{:=}&(n_{\sf Z})^{i\RR}(y),
\quad x,y\in \RR;
\end{split}
\end{equation}
${\sf Z}\subset \CC$ {\it конечной верхней плотности\/} (при порядке 1), если 
\begin{equation}\label{fdenZ}
{\type_1}[{\sf Z}]\overset{\eqref{fden}}{:=}{\type_1}[n_{\sf Z}]<+\infty.
\end{equation}

Последовательности ${\sf Z}=\{{\sf z}_k\}\subset \CC$ сопоставляем 
{\it экспоненциальную систему $\Exp^{\sf Z}\subset \Hol_* (\CC)$
с последовательностью показателей\/} $\sf Z$: 
\begin{equation}\label{ExpZ} 
\begin{split}
\Exp^{\sf Z}:=\bigl\{ z\mapsto z^pe^{{\sf z}_kz}
\colon& z\in \CC,\;  p\in \NN_0, \bigr.\\
&\bigl.0\leq p\leq n_{\sf Z}\bigl(\{{\sf z}_k\}\bigr)-1\overset{\eqref{df:divmn}}{=}{\sf Z}({\sf z}_k)-1\bigr\}.
\end{split}
\end{equation} 

\subsubsection{Последовательности нулей голоморфных функций}
Для функции $f\in \Hol_*(D)$ в области $D\subset  {\CC}$ через $\Zero_f$ обозначаем последовательность точек в $D$,   дивизор которой в каждой точке $z\in D$ равен кратности нуля функции $f$ в этой точке. Последовательность  $\Zero_f$ называем {\it последовательностью нулей,\/}
или корней, функции $f\in \Hol_*(D)$,  перенумерованной  с учетом кратности. Для нулевой функции   в $D$ по определению ее дивизор $\Zero_0=\boldsymbol{+\infty}$ на $D$. Для любой последовательности точек
${\sf Z}\subset D$ по определению  ${\sf Z} \subset \Zero_0$. 
Функция $f \in \Hol (D)$ {\it обращается в нуль на\/}
${\sf Z}$, если  ${\sf Z} \subset \Zero_f$. При этом пишем $f({\sf Z})=0$.

Пусть $H \subset \Hol(D)$. Последовательность ${\sf Z}\subset D$ --- {\it последовательность нулей для\/} $H$, если существует $f\in H$ с  $\Zero_f= {\sf Z}$; ${\sf Z}$ --- {\it подпоследовательность нулей для\/} $H$, если существует функция $f\neq  0$  из $H$, для которой $ {\sf Z} \subset \Zero_f$.  
Если $H$ --- векторное подпространство, то подпоследовательность нулей для $H$
называем еще и {\it множеством,\/} или {\it последовательностью, неединственности для\/} $H$; ${\sf Z}$ --- {\it множество,\/} или {\it последовательность, единственности для\/} $H$, если из $f\in H$ и  $f({\sf Z})=0$ следует $f= 0$.

\section{Результаты П.~Мальявена и Л.\,А.~Рубела\\ и их развитие и обобщения}\label{ss_RM}
\setcounter{equation}{0} 

Произвольная последовательность точек на $\CC$ без предельных точек в $\CC$ называем \textit{комплексной последовательностью\/} на $\CC$.
Комплексная последовательность на $\CC$ \textit{положительная,} если ее носитель --- подмножество $\RR^+$. 

\subsection{Случай положительных последовательностей}\label{2_1} Следуя \cite[2]{MR}, \cite[22]{RC}, последовательности  ${\sf Z}=\{{\sf z}_k\}\subset \CC$ 
сопоставляем идеал 
\begin{equation}\label{FZf}
I({\sf Z}):=\bigl\{f\in \Hol (\CC)\colon
f({\sf Z})=0\bigr\}\subset \Hol (\CC)
\end{equation} 
в кольце $\Hol (\CC)$, а также идеал в кольце всех ц.ф.э.т.\footnote{В \cite[2]{MR} и  \cite[22]{RC} идеал $I^1({\sf Z})$ обозначен соотв. как $\mathcal F ({\sf Z})$ и $F({\sf Z})$.}
\begin{equation}\label{FZf1}
I^1({\sf Z})\overset{\eqref{FZf}}{:=}I({\sf Z})\cap 
\bigl\{f\in \Hol (\CC)\colon {\type_1}\bigl[\log |f|\bigr]\overset{\eqref{typev}}{<}+\infty \bigr\}.
\end{equation} 
Полагаем 
\begin{equation}\label{II1}
\begin{split}
I_*({\sf Z})&\overset{\eqref{FZf}}{:=}I({\sf Z})\setminus \{0\}=I({\sf Z})\cap \Hol_*(\CC),
\\ 
I_*^1({\sf Z})&\overset{\eqref{FZf1}}{:=}I^1({\sf Z})\setminus \{0\}=I^1({\sf Z})\cap \Hol_*(\CC).
\end{split}
\end{equation}
\begin{propos}[{\rm \cite{Levin56},\cite{Levin96}}]\label{pr2_1}
 $I_*({\sf Z})\neq \varnothing$ тогда и только тогда, когда последовательность $\sf Z$ не имеет предельных точек в\/ $\CC$ --- {\rm частный случай классической теоремы Вейерштрасса для $\CC$.} 

$I^1_*({\sf Z})\neq \varnothing$ тогда и только тогда, когда последовательность $\sf Z$ конечной верхней плотности при порядке $1$, 
т.\,е. выполнено \eqref{fdenZ}, ---{\rm  частный случай классической теоремы Адамара\,--\,Вейерштрасса.} 
\end{propos}
Для положительной последовательности ${\sf Z}:=\{{\sf z}_k\}\subset \RR^+_*$ в \cite[1]{MR}, \cite[22]{RC} определен ее \textit{характеристический логарифм\/} $l_{\sf Z}\colon
(\RR^+_*)_{+\infty}\to \RR^+_{+\infty}$ как\footnote{В \cite{MR}, \cite{RC} пределы суммирования даны как $0<{\sf z}_k< r $, но при рассмотрении последовательностей ${\sf Z}$ конечной верхней плотности это несущественно.} 
\begin{equation}\label{logZ}
l_{\sf Z}(r):=\sum_{0<{\sf z}_k\leq r } \frac{1}{{\sf z}_k}\overset{\eqref{nuR}, \eqref{LS}}{=}\int_0^r \frac1{x}
\dd n_{\sf Z}^{\RR}(x), \quad 0<r
\leq +\infty. 
\end{equation}
Там же  рассмотрен  \textit{разностный характеристический логарифм\/}
\begin{equation}\label{logZ-}
l_{\sf Z}(r,R) \overset{\eqref{logZ}}{:=} l_{\sf Z}(R)- l_{\sf Z}(r),\quad  0<r<R\leq +\infty,
\end{equation}
--- {\it логарифмическая  функция интервалов\/} $(r,R]\subset  (\RR_{+\infty}^+)_*$.

\begin{theoMR}[{\rm Мальявена\,--\,Рубела \cite[теорема 4.1]{MR}, \cite[22, основная теорема]{RC}}]\label{MR1} Пусть ${\sf Z}\subset \RR_*^+$ и
${\sf W}=\{{\sf w}_k\}\subset \RR_*^+$ --- две \underline{положительные} последовательности конечной верхней плотности\/ \eqref{fdenZ}. Тогда в обозначениях \eqref{FZf1}--\eqref{II1} следующие три утверждения эквивалентны: 
\begin{enumerate}[{\rm (i)}]
\item\label{fgi} для любой $g\overset{\eqref{II1}}{\in} I_*^1({\sf W})$ найдется $f\overset{\eqref{II1}}{\in} I^1_*({\sf Z})$ с ограничением 
\begin{equation}\label{fgiR}
\bigl|f(iy)\bigr|\leq \bigl|g(iy)\bigr|\quad \text{при  всех $y\in \RR$};
\end{equation}
\item\label{fgii} существует постоянная $C\in \RR$, для которой 
\begin{equation}\label{Zld}
l_{\sf Z}(r,R)\overset{\eqref{logZ-}}{\leq} l_{\sf W}(r,R) 
+C\quad \text{при всех\/ $0<r<R<+\infty$};
\end{equation}
\item\label{fgiii} Существует пара функций $f\in I_*^1({\sf Z})$ и  $g\in I^1_*({\sf W})$ c $\Zero_g \cap\, \CC_{\rght} ={\sf W}$, т.\,е. с сужением  $n_{\Zero_g}\bigm|_{\CC_{\rght}}=n_{\sf W}$,   с ограничением  \eqref{fgiR}. 
\end{enumerate}
\end{theoMR}

Пусть $S\subset \CC$. \textit{Опорной функцией множества\/} $S$ называют
три расширенные числовые функции, однозначно определяющие друг друга:
\begin{enumerate}[{[sf1]}]
\item\label{sf1}   \textit{ выпуклую положительно однородную\/} функцию на  $\CC$:
\begin{equation}\label{Spf}
\Spf_S(z):=\sup_{s\in S} \Re s\bar z, \quad z\in \CC;
\end{equation}

\item\label{sf2}  сужение предыдущей функции на единичную окружность
 $\partial \DD$:
\begin{equation}\label{spfS}
\Spf_S\bigm|_{\partial \DD}(e^{i\theta}),
 \quad e^{i\theta} \in \partial \DD, \theta \in \RR;
\end{equation}
\item\label{sf3} \textit{$2\pi$-периодическую тригонометрически выпуклую\/} функцию на $\RR$:
\begin{equation}\label{spfR}
\spf_S(\theta ):=\sup_{s\in S}se^{-i\theta}, \quad \theta\in \RR. 
\end{equation}
\end{enumerate}
В теории функций одной комплексной переменной традиционно большей частью используется версия [sf\ref{sf3}] опорной функции. Функция $\spf_S$ принимает значение $-\infty$, если и только если  $S=\varnothing$. Функция $\spf_S$ принимает значение $+\infty$, если и только если множество $S$ неограниченно в $\CC$. 

{\it Шириной  множества\/ $S$ в направлении\/} $\theta$ называется величина 
\begin{equation}\label{widht}
\width_S(\theta):=\spf_S(\theta+\pi/2)+\spf_S(\theta-\pi/2) 
\end{equation} 
--- наименьшая ширина замкнутых в $\CC$ полос, включающих в себя $S$,  со сторонами, параллельными {\it замкнутому лучу}
\begin{equation}\label{ray}
\text{ray}(\theta):=\{te^{i\theta}\colon t\in \RR^+ \}, 
\end{equation}  
или {\it направлению\/} $\theta$. Для  $b\in \RR_{+\infty}^+$, $\theta \in \RR$  
\begin{equation}\label{strip}
\strip_b:=\bigl\{ z\in \CC \colon |\Im z|<b\bigr\},\quad
\overline{\strip}_b:=\bigl\{ z\in \CC \colon |\Im z|\leq b\bigr\}
\end{equation}
--- соотв. {\it открытая\/} и \textit{замкнутая в\/ $\CC$ полосы ширины\/} 
\begin{equation}\label{wstrip}
2b=\width_{{\strip}_b}(0)=\width_{\overline{\strip}_b}(0)
\end{equation} 
в направлении $0$ и со сторонами, параллельными $\RR$, 
 со средней линией $\RR$. 
Кроме того, в подразделе \ref{2_2} будет использован прямоугольник 
\begin{equation}\label{Pi}
\Pi(a,b)\overset{\eqref{strip}}{:=}
\overline{\strip}_b\cap i\,\overline{\strip}_a,
\quad \begin{cases}
\width_{\Pi (a,b)}(0)&=2b, \\
\width_{\Pi (a,b)}(\pi/2)&=2a. 
\end{cases}
\end{equation}  
При $g(z):=\sin \pi bz$ и ${\sf W}=\frac{1}{b}\NN$ из теоремы MR\ref{MR1} 
 вытекает
\begin{corolMR}[{\rm \cite[теорема 6.1, следствие 9.1]{MR}, \cite[теорема 22.1]{RC}}]\label{corolMR1} Пусть  $b\in \RR^+$  и ${\sf Z}\subset \RR_*^+$. Следующие три утверждения эквивалентны:

\begin{enumerate}[{\rm (i)}]
\item\label{bi} существует  $f\in I_*^1({\sf Z})$, удовлетворяющая ограничению
\begin{equation}\label{fbh}
\bigl|f(iy)\bigr|\leq \exp \pi b |y|\quad\text{при всех $y\in \RR$};
\end{equation}
\item\label{bii} существует постоянная $C\in \RR$, для которой
\begin{equation}\label{Zlds}
l_{\sf Z}(r,R)\overset{\eqref{logZ-}}{\leq}
  b\log \frac{R}{r}+C\quad  \text{при всех\/ $0<r<R<+\infty$};
\end{equation}
\item\label{biii} экспоненциальная система\footnote{В 
 \cite[следствие 9.1]{MR} --- система $\Exp^{-{\sf Z}}$, что  в данном случае можно не различать.} $\Exp^{\sf Z}$ с последовательностью показателей $\sf Z$ в обозначении\/ \eqref{ExpZ} для любого или некоторого $s\in \CC$ не полна в пространстве функций, непрерывных в  замкнутой полосе 
$s+\overline{\strip}_{\pi b}$ и голоморфных в открытой полосе $s+{\strip}_{\pi b}$, в  топологии равномерной сходимости на компактах из $s+\overline{\strip}_{\pi b}\subset \CC$.
\end{enumerate}
\end{corolMR}
При более слабых требованиях на ограничение роста функции $f\in I_*^1({\sf Z})$, чем  в теореме MR\ref{MR1} и в следствии MR\ref{corolMR1}, с начала 1960-х гг. известно также 

\begin{corolMR}[{\rm \cite[теоремы 6.3--5]{MR}}]\label{corolMR2} Пусть  ${\sf Z}=\{{\sf z}_k\}\subset \RR_*^+$ ---  последовательность конечной верхней плотности и задано число  $b\in \RR^+$. Следующие два утверждения эквивалентны: 
\begin{enumerate}[{\rm (i)}]
\item\label{bi0} существует  функция $f\in I_*^1({\sf Z})$, для которой 
${\ind_1}[\log |f|](\pm\pi/2)\leq \pi b$;
\item\label{bii0}существуют  $q\colon {\RR}^+\to \RR^+$ с $\lim\limits_{R\to +\infty} q(R)=0$ и  $C\in \RR$, для которых 
\begin{equation}\label{lZR+0}
l_{\sf Z} (r, R) \overset{\eqref{logZ-}}{\leq} \bigl(b+q(R)\bigr) \log \frac{R}{r}+C  \quad\text{при всех\/ $0<r<R<+\infty$}.
\end{equation}
\end{enumerate}
\end{corolMR}

Следуя \cite{MR}, рассмотрим следующие плотности:
\begin{subequations}\label{dens}
\begin{align}
\overline{\dens}_{\log}({\sf Z}) &:=\limsup_{1< a\to +\infty}\frac1{\log a}
\limsup_{r\to +\infty} l_{{\sf Z}}(r, ar) ,
\tag{\ref{dens}l}\label{densl}  \\
 {\dens}_{\inf\log}({\sf Z}) &:=\inf_{a> 1}\frac1{\log a}
\limsup_{r\to +\infty} l_{{\sf Z}}(r, ar) , 
\tag{\ref{dens}i}\label{densi}  \\
{\dens}_{\rm b\log}({\sf Z}) &:=\inf b , \tag{\ref{dens}b} \label{densb}
\end{align}
\end{subequations}
где последняя точная  нижняя граница $\inf$ берется по всем таким $b\in \RR_*^+$, что найдется постоянная $C_b\in \RR$, для которой 
\begin{equation}\label{estb}
l_{{\sf Z}}(r, R)\le b\log \frac{R}{r}+C_b
\quad\text{при всех $0\le r<R<+\infty$.}
\end{equation}
В \cite[3]{MR} показано, что для  последовательности ${\sf Z}$ {\it конечной верхней плотности\/} они совпадают и внешний верхний предел $\limsup$ в правой части \eqref{densl} можно заменить на обычный предел $\lim$. 
Все они в этом случае называются \textit{логарифмической блок-плотностью\/} последовательности ${\sf Z}$. 

\begin{corolMR}[{\cite[теорема 6.2, следствие 9.1]{MR}}] \label{corolMR3} 
 Пусть  ${\sf Z}=\{{\sf z}_k\}\subset \RR_*^+$ и $b\in \RR_*^+$. Следующие три утверждения эквивалентны:
\begin{enumerate}[{\rm (i)}]
\item\label{bie} 
существует функция  $f\in I_*^1({\sf Z})$, для которой ${\ind_1}[f](\pm \pi/2)<\pi b$;
\item\label{biie} ${\dens}_{\log}({\sf Z}) <b$;
\item\label{biiie}  система $\Exp^{\sf Z}$ не полна в пространстве  $\Hol(S_{\pi b})$.
\end{enumerate}
\end{corolMR}
В \cite{MR} установлены  и иные весьма содержательные результаты, например, об аналитическом продолжении рядов \cite[теоремы 9.2]{MR} и др.

\subsection{Случай комплексных последовательностей}\label{2_2}
Первые результаты, позволяющей перейти в теореме Мальявена\,--\,Рубела к {\it произвольным комплексным последовательностям\/}  $\sf Z\subset \CC$ и 
${\sf W}\subset \CC$ были получены в конце 1980-х гг. в работах  \cite{Kha88}--\cite{Kha89} первого из соавторов.    
Пусть ${\sf Z}=\{{\sf z}_k\}\subset \CC$
--- комплексная последовательность. Развивая \eqref{logZ}, определим \textit{ правый и левый характеристические логарифмы} последовательности ${\sf Z}$ как
\begin{subequations}\label{logZC}
\begin{align}
l_{{\sf Z}}^{\rght}(r)&\overset{\eqref{logZ}}{:=}\sum_{\substack{0 < |{\sf z}_k|\leq r\\{\sf z}_k \in \CC_{\rght}}} \Re \frac{1}{{\sf z}_k}, \quad 0<r\leq +\infty, 
\tag{\ref{logZC}$\mathrm l^+$}\label{df:dD+}\\
l_{{\sf Z}}^{\lft}(r)&\overset{\eqref{logZ}}{:=}\sum_{\substack{0< |{\sf z}_n|\leq r\\{\sf z}_k \in {_{\lft}\!}\CC}} \Re \frac{-1}{{\sf z}_k}, \quad 0<r\leq +\infty,
\tag{\ref{logZC}$\mathrm l^-$}\label{df:dD-}
\end{align}
\end{subequations}
а также  {\it правую\/} и {\it левую  логарифмические функции интервалов}   \cite{KhaD88}--\cite[\S~1]{kh91AA} 
\begin{subequations}\label{df:l}
\begin{align}
l_{{\sf Z}}^{\rght}(r, R)&\overset{\eqref{df:dD+}}{:=}
l_{{\sf Z}}^{\rght}(R)-l_{{\sf Z}}^{\rght}(r),
\quad 0< r < R \leq +\infty,
\tag{\ref{df:l}$\mathrm l^+$}\label{df:dDl+}\\
l_{{\sf Z}}^{\lft}(r, R)&\overset{\eqref{df:dD-}}{:=}l_{{\sf Z}}^{\lft}(R)-l_{{\sf Z}}^{\lft}(r), \quad 0< r < R \leq +\infty,
\tag{\ref{df:l}$\mathrm l^-$}\label{df:dDl-}
\\
\intertext{которые порождают {\it логарифмическую функцию  интервалов\/} для  $\sf Z$:} 
l_{{\sf Z}}(r, R)&\overset{\eqref{logZ-}}{:=}\max \{ l_{{\sf Z}}^{\lft}(r, R), l_{{\sf Z}}^{\rght}(r,
R)\}, \quad 0< r < R \leq +\infty .
\tag{\ref{df:l}l}\label{df:dDlL}
\end{align}
\end{subequations}
Для ${\sf Z}=\varnothing$ по определению $l_{\sf Z}(r,R)\equiv 0$ при всех 
$0< r < R \leq +\infty$.

В начале 1990-х гг. в работе первого из соавторов \cite{kh91AA} удалось достичь уровня ограничения \eqref{fgiR} для последовательностей \textit{ комплексных\/} точек при некотором дополнительном условии. Говорим, что последовательность 
${\sf Z} =\{ {\sf z}_k\} \subset \CC$ {\it отделена от мнимой оси,\/}  
 если 
\begin{equation}\label{con:dis}
\left(\liminf_{k\to\infty}\frac{\bigl| \Re {\sf z}_k \bigr|}{ |{\sf z}_k |} >0 \right)
 \Longleftrightarrow
\left(\limsup_{k\to\infty}\frac{\bigl| \Im {\sf z}_k \bigr|}{ |{\sf z}_k |} <1 \right).
\end{equation}
Пара эквивалентных ограничений \eqref{con:dis} геометрически означает, что  найдется  пара непустых открытых вертикальных углов, содержащих $i\RR_*$, для которой точки ${\sf z}_k$ лежат вне этой пары углов  при всех больших $|k|$. 

\begin{theoKh}[{\rm \cite[основная теорема]{kh91AA}--\cite[3.2.2]{Khsur}}]\label{theoKh1} 
Пусть комплексные последовательности ${\sf Z}\subset \CC$ и ${\sf W}\subset \CC_{\rght}$ отделены от $i\RR$. Тогда утверждения 
\eqref{fgi}--\eqref{fgiii} теоремы\/ {\rm MR\ref{MR1}} по-прежнему попарно эквивалентны. 
\end{theoKh}

\begin{corolKh}[{\rm \cite[основная теорема]{kh91AA}--\cite[3.2.2]{Khsur}}]\label{corolKh1}  Пусть  $b\in \RR^+$  и комплексная последовательность ${\sf Z}\subset \CC$ отделена от $i\RR$. Тогда утверждения \eqref{bi}--\eqref{biii} следствия {\rm MR\ref{corolMR1}} попарно эквивалентны.  
\end{corolKh}

\begin{corolKh}[{\rm \cite[следствие 4.1]{kh91AA}--\cite[теорема 3.2.5, следствие 3.2.1]{Khsur}}]\label{corolKh2}
 Пусть   $b\in \RR^+$ и \textit{комплексная\/} последовательность ${\sf Z}=\{{\sf z}_k\}\subset \CC$  отделена от $i\RR$. 
Тогда каждое из утверждений \eqref{bi0} и \eqref{bii0} следствия \/ {\rm MR\ref{corolMR2}}  эквивалентно одному и тому же утверждению 
\begin{enumerate}
\item[{\rm (iii)}]\label{Pi_iii} существует число $D\in \RR^+$, для которого  система $\Exp^{\sf Z}$ не полна в пространстве ростков $\Hol (s+\Pi(D,\pi b))$ на любом или некотором сдвиге $s+\Pi(D,\pi b)$  прямоугольника $\Pi (D,\pi b)$ из \eqref{Pi}, $s\in \CC$.
\end{enumerate}
\end{corolKh}
Различные виды логарифмических плотностей для \textit{комплексной} последовательности ${\sf Z}$ определяем  в точности как в \eqref{dens}--\eqref{estb}  на основе
логарифмической функции интервалов \eqref{df:dDlL}. Для последовательности конечной плотности эти плотности совпадают 
\cite[\S~1]{Kha89}--\cite[теорема 1]{KKh00} и по-прежнему называем их  \textit{логарифмической  блок-плотностью\/} уже \textit{ ком\-п\-л\-е\-к\-с\-н\-ой} последовательности ${\sf Z}$.
\begin{corolKh}\label{corolKh3} Пусть $b\in \RR_*^+$ и ${\sf Z}\subset \CC$ --- комплексная последовательность. Тогда утверждения \eqref{bie}--\eqref{biiie} из следствия\/ {\rm MR3} по-прежнему попарно эквивалентны, если в п.~\eqref{biie} добавить требование конечности верхней плотности последовательности  ${\sf Z}$. 
\end{corolKh}

Следующий исторически начальный результат конца 1980-х гг. при переходе от положительных к комплексным последовательностям точек не содержит, как и предшествующее следствие Kh\ref{corolKh3}, никаких дополнительных ограничений на последовательности, кроме естественного, в свете 
частного  случая классической теоремы Адамара\,--\,Вейерштрасса
из предложения \ref{pr2_1}, условия конечности верхней плотности последовательностей. Он может быть получен с помощью \cite[основная теорема]{Kha01l} из теоремы Kh\ref{theoKh1}. 
\begin{theoKh}[{\rm \cite[основная теорема]{KhaD88}, \cite[основная теорема]{Kha89}, \cite[теорема 3.2.1]{Khsur}}]\label{theoKh2}
Пусть комплексные последовательности ${\sf Z}\subset \CC$ и ${\sf W}\subset \CC_{\rght}$ конечной верхней пло\-т\-н\-ости при порядке $1$. Тогда следующие три утверждения   эквивалентны:
\begin{enumerate}[{\rm (i)}]
\item\label{cxi} для любой функции $g\overset{\eqref{II1}}{\in} I^1_*({\sf W})$ при любом  $\e \in \RR_*^+$ найдется 
функция $f_{\e}\overset{\eqref{II1}}{\in} I^1_*({\sf Z})$, для которой 
\begin{equation}\label{fgiRKh}
\log \bigl|f_{\e}(iy)\bigr|\leq \log \bigl|g(iy)\bigr|
+\e|y|\text{ при  всех $y\in \RR\setminus E_{\e}$}, 
\quad \lambda_{\RR}(E_{\e})<+\infty, 
\end{equation}
т.\,е. подмножество $E_{\e}\subset \RR$ конечной лебеговой меры на $\RR$;
\item\label{cxii} для любого $\e \in \RR_*^+$ найдется постоянная $C \in \RR^+$, для которой
\begin{equation}\label{lZZ}
l_{\sf Z} (r, R) \overset{\eqref{df:l}}{\leq} l_{\sf W} (r, R ) 
+\e \log \frac{R}{r}+C \quad\text{при всех\/ $0<r<R<+\infty$};
\end{equation}
\item\label{cxiii} Существует функция $g\in I^1_*({\sf W})$ 
c $\Zero_g \cap\, \CC_{\rght} ={\sf W}$, для которой при любом  $\e \in \RR_*^+$ найдется $f_{\e}\in I^1_*({\sf Z})$ с ограничением 
\eqref{fgiRKh}.
\end{enumerate}
\end{theoKh}
\begin{remark}\label{remi}
Теоремы Kh\ref{theoKh1}, Kh\ref{theoKh1} и следствия Kh\ref{corolKh1}--Kh\ref{corolKh3} доказываются в первоисточниках \cite{Kha88}--\cite{KhDD92}  ad ovo без  использования теоремы Мальявена\,--\,Рубела  MR\ref{MR1} и следствий MR\ref{corolMR1}--MR\ref{corolMR3}, хотя, безусловно, ряд  идей доказательств  почерпнуты  из основополагающей работы П. Мальявена и Л.\,А. Рубела \cite{MR}. К таковым относится и идея применения выметания на мнимую ось, но уже не только меры или заряда с носителем на положительной полуоси $\RR^+$, но и мер и зарядов, <<размазанных>> по всей плоскости $\CC$.      
\end{remark}

\section{Субгармонические результаты}\label{111}

Сформулированные в \S~\ref{111} результаты не только распространяют результаты 
из \S~\ref{ss_RM} на меры и субгармонические функции в $\CC$, 
но и в значительной мере развивают и усиливают их и в традиционном варианте применительно к последовательностям точек/нулей в $\CC$ и целым функциям.  

\subsection{\bf Виды и плотности зарядов и мер}\label{2_3_1} 
Распространим логарифмические функции интервалов
для ${\sf Z}\subset \CC$  из  \eqref{df:l}  на заряды  $\mu\in \Meas (\CC)$:

\begin{subequations}\label{df:lm}
\begin{align}
l_{\mu}^{\rght}(r, R)&\overset{\eqref{df:dDl+}}{:=}\int_{\substack{r < | z|\leq R\\ \Re z >0}} 
\Re \frac{1}{ z} \dd \mu(z), \quad 0< r < R \leq+\infty ,
\tag{\ref{df:lm}+}\label{df:dDlm+}\\
l_{\mu}^{\lft}(r, R)&\overset{\eqref{df:dDl-}}{:=}\int_{\substack{r< |z|\leq R\\ 
\Re z<0}}
\Re \Bigl(-\frac{1}{ z}\Bigr) \dd \mu(z) ,  \quad 0< r < R \leq +\infty ,
\tag{\ref{df:lm}-}\label{df:dDlm-}\\
\intertext{которые порождают {\it логарифмическую функцию интервалов\/} для $\mu$:}
l_{\mu}(r, R)&\overset{\eqref{df:dDlL}}{:=}\max \{ l_{\mu}^{\lft}(r, R), l_{\mu}^{\rght}(r,
R)\}, \quad 0< r < R \leq +\infty .
\tag{\ref{df:lm}m}\label{df:dDlLm}
\end{align}
\end{subequations}

Введем  также некоторые считающие радиальные функции с весом. В  \cite[теорема A]{Kha99} они рассматривались в значительно более общей форме для последовательностей точек и множеств в $\CC^n$. Для комплексных чисел $z\in \CC$ главные значения аргументов $\arg z$ нам удобнее выбирать здесь из интервала $[-\pi/2, 3\pi /2)$.

\begin{definition}
Пусть $\mu\in \Meas(\CC)$.  Введем в рассмотрение \textit{считающую  функцию заряда\/  $\mu$ с весом $k\colon [-\pi/2, 3\pi /2) \to \RR$\/}:
 \begin{equation}\label{c)}
\mu(r;k):=\int_{\overline D(r)} k (\arg z)\dd \mu(z), 
\end{equation}
где $k$ --- расширенная числовая борелевская функция  на $[-\pi/2, 3\pi /2)$. 
\end{definition}
В частных случаях 
\begin{equation}\label{kcos}
k(\arg z):=\cos^+\arg z, 
\quad k(\arg z):=\cos^-\arg z, \quad z\in \CC_*,
\end{equation}
 из определений \eqref{df:lm} в обозначениях \eqref{c)}--\eqref{kcos}  при любых значениях  $0< r < R <+\infty$  интегрированием по частям получаем 
\begin{subequations}\label{l_mu}
\begin{align}
l_{\mu}^{\rght}(r, R)&\overset{\eqref{df:dDlm+}}{=}
\int_r^R \frac{\dd \mu(t;\cos^+)}{t}
\notag\\
&=\frac{\mu (R;\cos^+)}{R}-\frac{\mu (r;\cos^+)}{r}
+\int_r^R\frac{\mu (t;\cos^+)}{t^2} \dd t, 
\tag{\ref{l_mu}+}\label{l_mu_m+}
\\
l_{\mu}^{\lft}(r, R)&\overset{\eqref{df:dDlm-}}{=}
\int_r^R \frac{1}{t} \dd \mu(t;\cos^-)
\notag \\
&=\frac{\mu (R;\cos^-)}{R}-\frac{\mu (r;\cos^-)}{r}
+\int_r^R\frac{\mu (t;\cos^-)}{t^2} \dd t.  
\tag{\ref{l_mu}-}\label{l_mu_m-}
\end{align}
\end{subequations}
Согласно соотношениям \eqref{df:dDlLm}, \eqref{l_mu} очевидно следующее  
\begin{propos}\label{pr:lm}
Для $\mu\in \Meas(\CC)$  {\it конечного типа}
${\type_1}[\mu]\overset{\eqref{nuR}}{<}+\infty$ для логарифмической функции интервалов  
\begin{equation}\label{ml}
\overline l_{\mu}(r,R):=\max \left\{\int_r^R\frac{\mu (t;\cos^-)}{t^2} \dd t, \int_r^R\frac{\mu (t;\cos^+)}{t^2} \dd t\right\}
\end{equation}
при любом $r_0\in \RR_*^+$ имеет место соотношение 
\begin{equation}\label{mlO}
\bigl|\overline l_{\mu}(r,R) - l_{\mu}(r,R)\bigr |=O(1), \quad r_0\leq r<R<+\infty.
\end{equation}
\end{propos} 
\begin{remark}\label{rem:lm}
Для произвольной последовательности ${\sf Z}\subset \CC$ со считающей мерой $n_{\sf Z}$ из \eqref{df:divmn} полагаем 
\begin{equation*}
\overline l_{\sf Z}^{\rght}(r, R)\overset{\eqref{l_mu_m+}}{:=}l_{n_{\sf Z}}^{\rght}(r, R), \;
\overline l_{\sf Z}^{\lft}(r, R)\overset{\eqref{l_mu_m-}}{:=}l_{n_{\sf Z}}^{\lft}(r, R), 
\; 
\overline l_{\sf Z}(r,R)\overset{\eqref{ml}}{:=}\overline l_{n_{\sf Z}}(r,R).
\end{equation*}
По предложению \ref{pr:lm} ниже всюду в разделе \ref{111} различные \textit{логарифмические функции интервалов\/} $l_{\mu}$, $l_{\sf Z}$  из  \eqref{df:l}, \eqref{df:lm}можно заменить  соотв. на \textit{логарифмические функции интервалов\/} $\overline l_{\mu}$ из \eqref{l_mu}, \eqref{ml} и $\overline l_{\sf Z}$.
\end{remark}
Заряд $\mu\subset \Meas (\CC)$ {\it отделен от мнимой оси\/} $i\RR$, если  
\begin{equation}\label{con:dism}
\left(\liminf_{\stackrel{z\to\infty}{z\in \supp \mu}}\frac{| \Re  z |}{ | z|} >0 \right)
\overset{\eqref{con:dis}}{\Longleftrightarrow}
\left(\limsup_{\stackrel{z\to\infty}{z\in \supp \mu}}\frac{|\Im  z|}{ | z |} <1 \right).
\end{equation}
Ограничение \eqref{con:dis} геометрически означает, что найдется пара непустых открытых вертикальных углов, содержащих $i\RR_*$, и число $R\in \RR^+_*$, для которых множество  $\supp \mu \setminus D(R)$ не пересекается с этой парой углов.

Заряд $\mu\in \Meas (\CC)$ удовлетворяет {\it условию Бляшке\/} как 
{\it в правой\/ $\CC_{\rght}$,} так  и {\it в левой\/ ${_{\lft}}\CC$ полуплоскостях,\/} если  \cite[4.1]{KhI}
\begin{equation}\label{cB}
\int_{\CC\setminus D(r_0)}\Bigl| \Re \frac1{z}\Bigr| \dd |\mu|(z)<+\infty
\quad \text{для некоторого $r_0\in \RR_*^+$}.
\end{equation}
При условии \eqref{cB}  заряд $\mu$ называем иначе зарядом, удовлетворяющий \textit{условию Бляшке в\/ $\CC\setminus i\RR$. } В терминах 
логарифмических функций интервалов \eqref{df:lm}, \eqref{ml} для заряда $\mu$ конечного типа по предложению \ref{pr:lm} условие Бляшке 
\eqref{cB} эквивалентно  соотношениям 
\begin{equation}\label{Oy}
\overline l_{|\mu|}(r,R)\overset{\eqref{mlO}}{=} l_{|\mu|}(r,R)+O(1)\overset{\eqref{cB}}{=}O(1), 
\quad 0<r_0\leq r<R<+\infty. 
\end{equation}

Пусть $I:=[iy_1,iy_2]$ --- замкнутый интервал на $i\RR$, $-\infty<y_1<y_2<+\infty$. {\it Двустороннюю гармоническую меру для\/ $\CC\setminus i\RR$ в точке\/ $z\in \CC$} обозначаем как  функцию интервалов 
\begin{equation}\label{omega}
\omega (z,[iy_1,iy_2])\overset{\text{\cite[3.1]{KhI}}}{:=}\omega_{\CC\setminus i\RR}(z,I):=\frac1{\pi}
\int_{y_1}^{y_2}\Bigl|\Re \frac{1}{iy-z}\Bigr| \dd y ,
\end{equation} 
равную делённому на $\pi$ углу, под которым виден интервал $I$ из точки $z\in \CC$ \cite[(3.1)]{kh91AA}, \cite[1.2.1, 3.1]{KhI}.
Если мера  $\mu \in \Meas^+(\CC)$ конечного типа удовлетворяет условию Бляшке в $\CC\setminus i\RR$, то  определено классическое выметание $\mu^{\bal}$ (рода $0$ \cite[определение 3.1]{KhII}) этой меры на мнимую ось \cite[следствие 4.1, теорема 4]{KhI}, задаваемое через функцию распределения $(\mu^{\bal})^{i\RR}$ из 
\eqref{nuiR}  меры $\mu^{\bal}$  с носителем $\supp \mu \subset i\RR $  как 
\begin{equation}\label{mubal}
(\mu^{\bal})^{i\RR}(iy)
\overset{\eqref{omega}}{:=}\begin{cases}
\int\limits_{\CC} \omega\bigl(z, [0,iy]\bigr)\dd \nu(z)\quad &\text{при $y\geq 0$},\\
-\int\limits_{\CC} \omega\bigl(z, [iy,0)\bigr)\dd \nu(z)\quad &\text{при $y< 0$},
\end{cases} 
\end{equation}

Заряд $\mu\in \Meas (\CC)$ \textit{конечной внешней плотности Кахана,\/} если 
$\mu$ \textit{конечной верхней плотности,}   удовлетворяет  \textit{условию Бляшке\/ \eqref{cB} в\/ $\CC\setminus i\RR$\/}
и существует \textit{липшицева функция} 
\begin{equation}\label{Lipk}
k\colon i\RR \to \RR, \quad l_k:=\sup_{\stackrel{y,y'\in \RR}{y\neq y'}}
\frac{|k(iy)-k(iy')|}{|y-y'|}<+\infty, 
\end{equation}
для которой конечен \textit{логарифмический интеграл\/} \cite{Koosis88}, \cite{Koosis92}, \cite{Koosis96} 
\begin{equation}\label{J2mu}
J\bigl(|(\mu^{\bal})^{i\RR}-k|\bigr)
\overset{\eqref{fK:abp+}}{=}\Bigl(\int_{-\infty}^{-1}+
\int^{+\infty}_{1}\Bigr)\frac{\bigl|(\mu^{\bal})^{i\RR}(iy)-k(iy)\bigr|}{y^2}\dd y
<+\infty.
\end{equation}
Наряду с интегралом \eqref{J2mu} используем и интегралы из  \cite [4.6, (4.16)]{KhII}:
\begin{equation}\label{fK:abp}
J(r,R; v):=
J_{-\frac{\pi}{2},\frac{\pi}{2}}^{[1]}(r, R;v)
\overset{\text{ \cite[4.6,(4.16)]{KhII}}}{:=}
\int_r^{R} \frac{v(-iy)+v(iy)}{y^2} \dd y
\end{equation}
при $0<r<R\leq +\infty$. Заряд $\nu\in \Meas(\CC)$ 
удовлетворяет\footnote{Это условие, вообще говоря, действительно слабее, чем выполнение условия Бляшке из \eqref{cB} одновременно в правой и левой полуплоскостях $\CC_{\rght}$ и ${_{\lft}}\CC$.}  
{\it слабому двустороннему условию Бляшке относительно мнимой оси\/ $i\RR$ рода\/ $1$} \cite[определение 2.2.1]{kh91AA}, \cite[определение 2.2.1]{KhDD92}, \cite[(3.23)]{KhII},  если  при некотором $r_0>0$
\begin{equation}\label{cB2}
\left| \int_{D(r)\setminus D(r_0)}\Bigl|\Re\frac{1}{z} \Bigr| \dd \nu(z)\right|=O(1), \quad r\to +\infty,
\end{equation}
Заряд $\nu\in \Meas(\CC)$ удовлетворяет  {\it условию Линделёфа рода\/ $1$}  \cite[\S~2]{kh91AA}, \cite[определение 2.3.1]{KhDD92}, \cite[определение 4.7]{KhI}, если  для некоторого $r_0\in \RR^+_*$
\begin{equation}\label{con:Lp+} \Biggl|\;\int_{D(r)\setminus D(r_0)} \frac{1}{z}
\dd \nu(z) \Biggr|=O(1) \quad\text{при $r\to +\infty$,} 
\end{equation} 
\begin{remark}\label{remL}
Очевидно, для меры $\nu\in \Meas^+(\CC)$, удовлетворяющей условию Линделёфа   
\eqref{con:Lp+} рода $1$,  при всех  $0<r_0\leq r<R<+\infty$ имеем  
\begin{equation}\label{lL}
l_{\nu}^{\rght}(r,R)=l_{\nu}^{\lft}(r,R)+O(1)=l_{\nu}(r,R)+O(1).
\end{equation}
Если $\upsilon\in \Meas^+(\CC)$ --- мера Рисса субгармонической функции \textit{конечного типа,} то мера $\upsilon$ \textit{конечного типа\/} и удовлетворяет \textit{условию Линделёфа\/} рода $1$. Следовательно, для такой меры имеет место  \eqref{lL}, где по предложению \ref{pr:lm} функцию интервалов $l_{\nu}$  можно заменить   на  $\overline l_{\nu}$  из \eqref{l_mu}--\eqref{ml}.

Обратно, если $u\in \sbh_*(\CC)$ --- функция \textit{порядка\/} \cite[2.1, (2.1a)]{KhI} 
\begin{equation}\label{ord}
\ord_{\infty}[u]\overset{\eqref{df:MCBm}}{:=}\limsup\frac{\log \bigl(1+M_u^+(r)\bigr)}{\log r}\leq 1
\end{equation} 
с мерой Рисса $\upsilon_u\in \Meas^+(\CC)$  \textit{конечного типа,\/}
удовлетворяющей\textit{ условию Линделёфа\/} \eqref{con:Lp+} рода $1$, 
то\textit{ функция $u$ конечного типа,\/} что следует, например, из представления Вейерштрасса\,--\,Адамара для субгармонических функций конечного порядка \cite[гл.~4, 4.2]{HK}.
\end{remark}

\subsection{Основные теоремы и следствия}

\begin{theorem}[ср. с теоремой Kh\ref{theoKh1}]\label{th1}
Пусть заданы {\it меры\/} $\nu, \mu \in \Meas^+(\CC)$,
\begin{itemize}
\item[{\bf [meas]}]\label{nu} 
 представленные  в виде сумм мер
\begin{equation}\label{numu}
\begin{cases}
\nu&:=\nu_0+\nu_{\infty},\quad \nu_0, \nu_{\infty} \in \Meas^+(\CC),\\
\mu&:=\mu_0+\mu_{\infty},\quad  \mu_0, \mu_{\infty} \in \Meas^+(\CC),
\end{cases}
\end{equation} 
где меры\/ 
$\nu_{\infty}, \mu_{\infty}$  \underline{отделены  от мнимой оси}\/ $i\RR$ в смысле \eqref{con:dism}.
\end{itemize}
Кроме того, пусть 
\begin{equation}\label{suppm}
\type_1[\nu+\mu]\overset{ \eqref{fden}}{<}+\infty,
\quad  \supp \mu \subset \CC_{\overline \rght}, 
\end{equation} 
мера  $\mu_0$ удовлетворяет\underline{ условию Бляшке\/} \eqref{cB}  в\/ $\CC\setminus i\RR$, 
а мера  $\nu_0$ \underline{конечной} \underline{внешней плотности Кахана} в смысле \eqref{mubal}--\eqref{J2mu}. 

Тогда следующие три утверждения   эквивалентны.
\begin{enumerate}[{\rm {\bf a}1.}]
\item\label{a1}  Для \underline{любой} функции $M\in \sbh_*(\CC)$ \underline{конечного типа} $\type_1[M]\overset{\eqref{typev}}{<}+\infty$ с мерой Рисса $\mu_M\geq \mu$, для \underline{любой} функции $u\in \sbh_*(\CC)$ с мерой Рисса $\upsilon_u=\nu$,  для \underline{любого} числа $p\in \RR^+$ \underline{найдётся} такая \underline{целая функция}  $f\in \Hol_*(\CC)$, что 
\begin{subequations}\label{fgiRKh+0b}
\begin{align}
\type_1\bigl[u+\log |f|\bigr]<+\infty,&\quad \text{где, очевидно, $u+\log |f|\in \sbh_*(\CC)$,} 
\tag{\ref{fgiRKh+0b}t}\label{fgiRKh+0bt} 
\\
\intertext{и выполнены неравенства}
u(iy)+\log \bigl|f(iy)\bigr|
\overset{\eqref{df:MCBb}}{\leq}& B_{M}\Bigl(iy, \frac1{(1+|y|)^{p}}\Bigr)
\quad
\text{ для всех $ y\in \RR$}. 
\tag{\ref{fgiRKh+0b}i}\label{fgiRKh+0bi} 
\end{align}
\end{subequations}
\item\label{a2}\underline{ Существует}  постоянная $C\in \RR^+$, для которой 
\begin{equation}\label{lZZ+0b}
l_{\nu} (r, R) \overset{\eqref{df:lm}}{\leq} l_{\mu} (r, R ) +C  \quad \text{при всех\/ $1\leq r<R<+\infty$},
\end{equation}
где $l_{\nu}$ и/или $l_{\mu}$ можно заменить  соотв. на 
$\overline l_{\nu}$ и/или $\overline l_{\mu}$ из \eqref{l_mu}--\eqref{ml}.
\item\label{a3} \underline{Существуют}
\begin{enumerate}[{\rm i)}]
\item\label{b3i}  функция $M\in \sbh_*(\CC)$ \underline{конечного типа} с сужением  меры Рисса  \begin{equation}\label{muC}
\mu_M\bigm|_{\CC_{\rght}}\overset{\eqref{suppm}}{=}\mu, 
\end{equation}
\item\label{b3ii} функция $U\in \sbh_*(\CC)$ \underline{конечного типа} с  мерой  Рисса $\upsilon_U\geq \nu$, 
 \item\label{b3iii} число $y_0\in \RR^+$ и функция $q\colon i\RR_* \to \RR_*^+$, $q(i\cdot)\colon \RR_*\to \RR_*^+$ возрастает на $\RR_*^+$ и убывает на $-\RR_*^+$  с ограничением 
 \begin{equation}\label{Jq}
J(q)\overset{\eqref{fK:abp+}}{:=}\int_1^{+\infty}\frac{q(iy)+q(-iy)}{y^2}\dd y <+\infty,
\end{equation}
\end{enumerate}
 для которых  имеют место неравенства 
\begin{equation}\label{fgiRKh+b}
U(iy)\overset{\eqref{df:MCBc}}{\leq} C_{M}\bigl(iy, q(iy)\bigr)
+q(iy)\quad\text{при всех $ |y|\geq y_0$}. 
\end{equation}
\end{enumerate}
\end{theorem}

\begin{theorem}[ср. со  следствием Kh\ref{corolKh2}]\label{th2}
Пусть в соглашениях\/ {\bf[meas]} из теоремы\/ {\rm \ref{th1}} выполнено 
\eqref{suppm}, а также  ${\type_1}[\nu_0+\mu_0]\overset{\eqref{fden}}{=}0$. 

Тогда следующие три утверждения   эквивалентны.
\begin{enumerate}[{\rm {\bf b}1.}]
\item\label{b1} Для \underline{любой} функции $M\in \dsbh_*(\CC)$ \underline{конечного типа} с мерой Рисса $\mu_M \geq \mu$, для \underline{любой} функции $u\in \sbh_*(\CC)$ с мерой Рисса $\upsilon_u=\nu$, для \underline{любого} числа\/ $p\in \RR^+$
найдётся такая \underline{целая функция}  $f\in \Hol_*(\CC)$, 
что имеет место\/  \eqref{fgiRKh+0bt} и 
\begin{equation}\label{fgiRKh+0}
u(iy)+\log \bigl|f(iy)\bigr|\overset{\eqref{df:MCBb}}{\leq} B_{M}\Bigl(iy, \frac1{(1+|y|)^{p}}\Bigr)+o\bigl(|y|\bigr)\text{ при $ y\to \pm \infty$}. 
\end{equation}

\item\label{b2} \underline{Существуют}  функция 
\begin{equation}\label{m0}
m\colon {\RR}_*^+\to \RR_*^+, \quad \lim\limits_{r\to +\infty} m(r)=0,
\end{equation} 
и  постоянная $C\in \RR^+$, для которых 
\begin{equation}\label{lZZ+0}
l_{\nu} (r, R) \overset{\eqref{df:lm}}{\leq} l_{\mu} (r, R ) 
+m(r) \log \frac{R}{r}+C  \quad\text{при всех\/ $1\leq r<R<+\infty$},
\end{equation}
где $l_{\nu}$ и/или $l_{\mu}$ можно заменить  соотв. на 
$\overline l_{\nu}$ и/или $\overline l_{\mu}$ из \eqref{l_mu}--\eqref{ml}.

\item\label{b3}  \underline{Существуют} функции $M$, $U$ из пп.~{\bf a}{\rm\ref{b3i})--\ref{b3ii})}  теоремы {\rm\ref{th1}} и   функция \begin{equation}\label{q0}
q\colon i\RR_*\to\RR_*^+, \quad q(iy)=o(|y|)\quad\text{при $y\to \pm\infty$,} 
\end{equation}
для которых имеет место соотношение
\begin{equation}\label{fgiRKh+0C}
U(iy)\overset{\eqref{df:MCBc}}{\leq} C_{M}\bigl(iy, q(iy) \bigr)+o\bigl(|y|\bigr)\text{ при $ y\to \pm \infty$}. 
\end{equation}

\end{enumerate}
\end{theorem}
В следующей теореме \ref{th3} мы отказываемся от представления мер 
$\nu, \mu$ из   \eqref{numu} и соотв. от условий отделённости мер от мнимой оси из {\bf [meas]}.   

\begin{theorem}[{\rm ср. с теоремой Kh\ref{theoKh2}}]\label{th3}
Пусть   для мер  $\nu, \mu \in \Meas^+(\CC)$ имеет место  \eqref{suppm}.
 Тогда следующие три утверждения   эквивалентны.
\begin{enumerate}[{\rm {\bf c}1.}]
\item\label{c1} 
Для \underline{любой} функции $M\in \dsbh_*(\CC)$ \underline{конечного типа} с мерой Рисса $\mu_M \geq \mu$, для \underline{любой} функции $u\in \sbh_*(\CC)$ с мерой Рисса $\upsilon_u=\nu$, для  \underline{любых} чисел  $\e\in \RR^+_*$ и $p\in \RR^+$
найдётся такая \underline{целая функция}  $f\in \Hol_*(\CC)$, 
что имеет место\/  \eqref{fgiRKh+0bt} и  
\begin{equation}\label{fgiRKh+}
u(iy)+\log \bigl|f(iy)\bigr|\overset{\eqref{df:MCBb}}{\leq} B_{M}\Bigl(iy, \frac1{(1+|y|)^{p}}\Bigr)+\e|y|\text{ при  всех $y\in \RR$}. 
\end{equation}
\item\label{c2} Для \underline{любого} числа\/ $\e \in \RR_*^+$ найдется постоянная $C_\e\in \RR^+$, для которой
\begin{equation}\label{lZZ+}
l_{\nu} (r, R) \overset{\eqref{df:lm}}{\leq} l_{\mu} (r, R ) 
+\e \log \frac{R}{r}+C_\e \quad\text{при всех\/ $1\leq r<R<+\infty$},
\end{equation}
где $l_{\nu}$ и/или $l_{\mu}$ можно заменить  соотв. на 
$\overline l_{\nu}$ и/или $\overline l_{\mu}$ из \eqref{l_mu}--\eqref{ml}.
\item\label{c3}  
\underline{Существует} такая функция $M$ из п.~{\bf a}{\rm\ref{b3i})}  теоремы {\rm\ref{th1}}, что  при   \underline{любом} $\e\in \RR^+_*$  \underline{найдётся}     функция $U$ из  п.~{\bf a}{\rm\ref{b3ii})}  теоремы {\rm\ref{th1}},  для которой 
\begin{equation}\label{fgiRKh+C}
U(iy)\overset{\eqref{df:MCBc}}{\leq} C_{M}
\bigl(iy, \e |y|\bigr) +\e|y|\quad\text{ при  всех $y\in \RR$}. 
\end{equation}
\end{enumerate}
\end{theorem}

\section{Доказательства импликаций 
{\bf x}\ref{a1}${\Rightarrow}${\bf x}\ref{a3}${\Rightarrow}${\bf x}\ref{a2},  при 
$\bf x:=a,b,c$\\ для основных теорем \ref{th1}--\ref{th3}}
Знак вопроса ? над импликацией  или эквивалентностью 
в виде $\overset{\mathbf{?}}{\Rightarrow}$ или  $\overset{\mathbf{?}}{\Leftrightarrow}$
будет означать, что доказывается эта импликация  или эквивалентность.
 
\subsection{{\bf x}\ref{a1}$\overset{\mathbf{?}}{\Rightarrow}${\bf x}\ref{a3}}\label{1imp}
Для меры $\mu\in \Meas^+(\CC)$ с $\type_1[\mu]<+\infty$ 
при условии  \eqref{suppm} даже  \textit{\underline{без} представления} $\mu\overset{\eqref{numu}}{=}\mu_0+\mu_{\infty}$ из {\bf [meas]}
сразу следует существование функции $M\in \sbh_*(\CC)$ с $\type_1[M]<+\infty$ с  сужением меры Рисса $\mu_M\bigm|_{\CC_{\rght}}=\mu$, как это требуется в \eqref{muC}. Для этого достаточно сначала выбрать меру  $\mu_M:=\mu+\mu_-$, где 
$\mu_(S):=\mu(-S)$ для всех борелевских $S\subset \CC$, которая, очевидно, удовлетворяет условию Линделёфа \eqref{con:Lp+} рода $1$. 
После этого  можно  построить требуемую функцию $M$ конечного типа с помощью классического интегрального представления Вейерштрасса\,--\,Адамара   \cite[4.1--4.2]{HK} через заряд $\mu_M$ и  субгармоническое ядро Вейерштрасса\,--\,Адамара рода $1$ \cite[6.1]{KhI}. В качестве функции 
$q$ соотв. в {\bf a}\ref{a3}\ref{b3iii}) с \eqref{Jq}, в {\bf b}\ref{b3} с \eqref{q0}
выберем $q\equiv 1$ на $\RR^+$, а в {\bf c}\ref{c3} положим $q(iy)\equiv \e|y|$.
В качестве функции $U$ в {\bf a}\ref{a3}, {\bf b}\ref{b3}, 
{\bf c}\ref{c3} выберем функцию $u+\log|f|=:U$ соотв. из 
{\bf a}\ref{a1}\eqref{fgiRKh+0b}, {\bf b}\ref{b1}\eqref{fgiRKh+0}, {\bf c}\ref{c1}\eqref{fgiRKh+}. При таком выборе соотношения  
{\bf a}\ref{a1}\eqref{fgiRKh+0b}, {\bf b}\ref{b1}\eqref{fgiRKh+0}, {\bf c}\ref{c1}\eqref{fgiRKh+} влекут за собой соотв. соотношения 
{\bf a}\ref{a3}\eqref{fgiRKh+b}, {\bf b}\ref{b3}\eqref{fgiRKh+0C}, 
{\bf c}\ref{c3}\eqref{fgiRKh+C}, поскольку для любой функции
$M\in \sbh(\CC)$ имеем $B_M\overset{\eqref{df:MCB}}{\leq} C_M$ \cite[теорема 2.6.8]{Rans}.
\begin{remark}\label{rem1_3}
При доказательстве  импликаций  
{\bf x}\ref{a1}${\Rightarrow}${\bf x}\ref{a3} использованы лишь ограничения \eqref{suppm} на меры $\nu,\mu\in \Meas^+(\CC)$, а именно: 
\begin{enumerate}[1)]
\item\label{1_c} меры $\nu,\mu$ \textit{конечного типа при порядке\/} $1$, что  \textit{необходимо\/}  для существования  $u, M\in \sbh_*(\CC)$ с мерами Рисса  соотв. $\upsilon_u\geq \nu$, $\mu_M\geq \mu$; 
\item\label{2_c}  $\supp \mu\overset{\eqref{suppm}}{\subset} \CC_{\overline \rght}$, что фактически тоже \textit{необходимо\/} для существования функции $M$ из 
из п.~{\bf a}{\rm\ref{b3i})}  теоремы {\rm\ref{th1}}.
\end{enumerate}
Представления же {\bf [meas]}\eqref{numu} для мер $\nu,\mu\in \Meas^+(\CC)$
не применялись.  
\end{remark}
\subsection{{\bf x}\ref{a3}$\overset{\mathbf{?}}{\Rightarrow}${\bf x}\ref{a2}}\label{2imp}
Пусть выполнено  {\bf a}\ref{a3}, {\bf b}\ref{b3} или  {\bf c}\ref{c3}
 с функциями $q$ соотв. из  {\bf a}\ref{b3iii}) с ограничением \eqref{Jq}, из 
\eqref{q0} или вида $q(iy)=q_\e (iy)=\e |y|$, $y\in \RR$, с $\e\in \RR_*^+$. 
Не умаляя общности, в каждом из трёх случаев можно считать функцию $q(i\cdot)\colon \RR_*\to \RR^+$ гладкой,  возрастающей на $\RR^+$ и убывающей на $-\RR^+$ с ограничением  $q(iy)\leq a|y|=q_a(iy)$, $y\in \RR_*$ при некотором $a\in \RR_*^+$. Тогда открытое множество   
\begin{equation}\label{Oq}
O_q:=\bigl\{z\in \CC\colon -q(i\Im z)< \Re z< q(i\Im z) \bigr\}
\end{equation}
--- объединение областей $O_q^{\up}:=O_q\cap \CC^{\up}$ и  $
(O_q)_{\lw}:=O_q\cap \CC_{\lw}$ с  границей $\partial D_q$, позволяющей определить субгармоническую в $\CC$ функцию \cite[3.8]{HK}, \cite[теорема 2.4.5]{Rans}
\begin{equation}\label{Uq}
M_q=\begin{cases}
M\quad \text{\textit{\underline{на} дополнении\/} $\CC\setminus O_q$},\\
\text{\textit{гармоническое продолжение\/} $M$}\\
\text{\textit{внутрь\/  $O_q^{\up}$ и\/  $
(O_q)_{\lw}$ \underline{на}\/ $O_q$},}
\end{cases}
\end{equation}
не превышающую на $\CC$ функцию $M_{q_a}$ --- классическое выметание рода $0$
функции $M$ из двух открытых вертикальных углов раствора 
$<\pi$ с вершиной в нуле \cite[6.2, теорема 7]{KhI}. При этом функция $M_{q_a}$ конечного типа \cite[теорема 8]{KhI}, \cite[основная теорема]{Kha91}. Следовательно, и функция $M_q\leq M_{q_a}$ конечного типа. По построению    \eqref{Oq}--\eqref{Uq} имеем 
\begin{equation}\label{CMiy}
C_M\bigl(iy, q(iy)\bigr)\leq M_q(iy)\quad\text{для всех $y\in \RR$},
\end{equation}
а мера Рисса $\mu_{M_q}$ функции $M_q$ даёт  сужения
\begin{equation}\label{mud}
\mu_{M_q}\bigm|_{\CC\setminus O_q}=
\mu_M\bigm|_{\CC\setminus O_q}, \quad 
\mu_M^q:=\mu_M\bigm|_{O_q}.
\end{equation} 
Из условий-неравенств \eqref{fgiRKh+b}, \eqref{fgiRKh+0C}, \eqref{fgiRKh+C}
   и из \eqref{CMiy} в обозначении  \eqref{fK:abp}
\begin{multline}\label{JUM}
J(r,R;U)\overset{\eqref{fK:abp}}{=}\int_r^R\frac{U(iy)+U(iy)}{y^2} \dd y \\
\overset{ \eqref{fgiRKh+b}, \eqref{fgiRKh+0C}, \eqref{fgiRKh+C}}{\leq}  
\int_r^R\frac{C_M(iy)+q(iy)+C_M(-iy)+q(-iy)}{y^2} \dd y\\
 =J(r,R;C_M)+J(r,R;q)\overset{ \eqref{CMiy}}{\leq}
J(r,R;M_q)+J(r,R;q)
\end{multline}
при всех $1\leq r<R<+\infty$. 
\begin{lemma}[{\rm \cite[предложение 4.1, (4.19)]{KhII}}]\label{lemJ} 
Пусть функция $u\in \sbh_*(\CC)$ конечного типа с мерой Рисса $\upsilon$. Тогда 
при $1\leq r<R<+\infty$ имеем 
\begin{equation}\label{Jl}
\bigl|J(r,R;u)-l_{\upsilon}^{\rght}(r,R)\bigr|+\bigl|J(r,R;u)-l_{\upsilon}^{\lft}(r,R)\bigr|
=O(1). 
\end{equation}
\end{lemma}
Применение леммы \ref{lemJ} к каждой из функций $U$ и $M_q$
ввиду  \eqref{JUM} даёт
\begin{multline}\label{lUMq}
l_{\nu}(r,R)\leq l_{\upsilon_U}(r,R)\leq l_{\mu_{M_q}}(r,R)+
J(r,R;q)+O(1) \\
\overset{\eqref{mud}}{\leq}  
l_{\mu_M}(r,R)+l_{\mu_M^q}(r,R)+J(r,R;q)+O(1),\quad 1\leq r<R<+\infty. 
\end{multline}
Мера $\mu_{M}$ удовлетворяет условию Линделёфа рода $1$, поэтому по замечанию \ref{remL} из  \eqref{lL}, \eqref{lUMq} и \eqref{muC} следует  
\begin{multline}\label{lUMqm}
l_{\nu}(r,R)
\overset{\eqref{lL},\eqref{mud}}{=}
l_{\mu_M}^{\rght}(r,R)+l_{\mu_M^q}(r,R)+J(r,R;q)+O(1) \\
\overset{\eqref{muC},\eqref{suppm}}{=}
l_{\mu}(r,R)+l_{\mu_M^q}(r,R)+J(r,R;q)+O(1),\;1\leq r<R<+\infty,  
\end{multline}
где о мере $\mu_M^q\in \Meas^+(\CC)$ важно  помнить то, что 
\begin{equation}\label{muMq}
\type_1[\mu_M^q]<+\infty, \quad 
\supp \mu_M^q\overset{\eqref{mud},\eqref{Oq}}{\subset} \clos O_q. 
\end{equation}

\subsubsection{\bf Завершение {\bf a}\ref{a3}$\overset{\mathbf{?}}{\Rightarrow}${\bf a}\ref{a2}} Будет использована
\begin{lemma}\label{la2}
Пусть $q\colon i\RR_* \to \RR^+$ из  {\rm {\bf a}\ref{b3iii})} с ограничением \eqref{Jq}, 
\begin{equation}\label{eta1}
\eta\in \Meas^+(\CC), \quad \type_1[\eta]<+\infty, 
\quad\supp \eta \overset{\eqref{Oq}}{\subset} \clos O_q.
\end{equation}
Тогда мера $\eta$ удовлетворяет условию Бляшке \eqref{cB}--\eqref{Oy} в $\CC\setminus i\RR$.
\end{lemma}
\begin{proof}[Доказательство леммы \ref{la2}] По определению \eqref{df:dDlm+}
\begin{multline}\label{leta}
l_{\eta}^{\rght}(1,+\infty)\overset{\eqref{df:dDlm+}}{=} 
\int_{\substack{|z|>1,\\ \Re z >0}}\Re \frac{1}{ z} \dd \eta (z)\\
\overset{\eqref{eta1}}{\leq}  
\int_{\substack{| |z|\geq 1\\ \Re z >0}} 
\frac{q(i\Im z)+q(-i\Im z)}{|z|^2} \dd \eta(z)\\
\leq \int_{\substack{| |z|\geq 1\\ \Re z >0}} 
\frac{q(i|z|)+q(-i|z|)}{|z|^2} \dd \eta(z)\overset{\eqref{df:nup}}{=}
\int_1^{+\infty} \frac{q(iy)+q(-iy)}{y^2} \dd \eta^{\rad}(y) \\
=\int_1^{+\infty} \frac{Q(y)}{y^2} \dd \eta^{\rad}(y), \quad
Q(y):=q(iy)+q(-iy),
\end{multline}
где согласно {\rm {\bf a}\ref{b3iii})} и  \eqref{Jq}
 функция $Q$ возрастающая на $\RR_*^+$,  $J(Q)<+\infty$ и $Q(y)=o(y)$, $y\to +\infty$. Для правой части \eqref{leta} имеем
\begin{multline}\label{QI}
\int_1^{+\infty} \frac{Q(y)}{y^2} \dd \eta^{\rad}(y)=
-\int_1^{+\infty} {Q(y)} \dd \int_y^{+\infty}\frac1{t^2} \dd \eta^{\rad}(t)\\
=Q(1)\int_1^{+\infty}\frac1{t^2} \dd \eta^{\rad}(t)
+\int_1^{+\infty}  \int_y^{+\infty}\frac1{t^2} \dd \eta^{\rad}(t)  \dd {Q(y)}\\
\leq \const +\const \int_1^{+\infty} \frac{1}{y}\dd Q(y)
\leq \const +J(Q) <+\infty, 
\end{multline}
Из \eqref{leta}--\eqref{QI} имеем  $l_{\eta}^{\lft}(1,+\infty)<+\infty$.
 Аналогично для $l_{\eta}^{\lft}(1,+\infty)$. 
\end{proof}
Применение леммы \ref{la2} к мере $\eta=\mu_M^q$ из \eqref{muMq} даёт
$l_{\mu_M^q}(r,R)=O(1)$ для всех $1\leq r<R<+\infty$, откуда согласно  
\eqref{lUMqm} получаем 
\begin{multline}\label{lUMqml}
l_{\nu}(r,R)\overset{\eqref{lUMqm}}{\leq}
l_{\mu}(r,R)+O(1)+J(r,R;q)+O(1)\\
\overset{\eqref{Jq}}{\leq}
l_{\mu}(r,R)+O(1)+O(1)+O(1), \quad 1\leq r<R<+\infty,  
\end{multline}
что и доказывает импликацию 
{\bf a}\ref{a3}$\Rightarrow${\bf a}\ref{a2}. 

\subsubsection{\bf Завершение {\bf b}\ref{b3}$\overset{\mathbf{?}}{\Rightarrow}${\bf b}\ref{b2}} Будет использована
\begin{lemma}\label{lab}
Пусть функция $q$  вида   \eqref{q0}, 
a $\eta$ --- мера из \eqref{eta1}. Тогда 
найдутся функция $m$ вида 
\eqref{m0} и число  $C\in \RR$, с которыми 
\begin{equation}\label{eta0}
J(r,R;q)+l_\eta(r,R)\leq m(r)\log \frac{R}{r}+C\quad \text{для всех $1\leq r<R<+\infty$.}
\end{equation}
\end{lemma}
\begin{proof}[Доказательство леммы \ref{lab}]  Изменим функцию $q$, положив 
\begin{equation}\label{qy}
k(y):=\sup_{t\geq |y|} \frac{q(it)+q(-it)}{t}, \quad
q(iy):=k(y) |y|,  \quad y\in \RR_*, 
\end{equation}
где по построению \eqref{qy} функция  
\begin{equation}\label{k0}
k \text{ \it  четная и убывающая на\/ } \RR_*^+, \quad  
\lim_{y\to \pm\infty} k(y)=0.
\end{equation}
При этом условия вида \eqref{q0} и \eqref{eta1} для новой функции $q$ из 
\eqref{qy} сохраняются и  для всех $1\leq r<R<+\infty$ имеем 
\begin{equation}\label{k}
J (r,R;q)\overset{\eqref{qy},\eqref{k0}}{=}2\int_r^R \frac{k(y)|y|}{y^2} \dd y \leq 2k(r)\log\frac{R}{r}. 
\end{equation}
Для меры $\eta$ из  \eqref{eta1} подобно \eqref{leta} достаточно получить \eqref{eta0}
для правой логарифмической функции интервалов $l_{\eta}^{\rght}(r,R)$. 

В обозначении $\overline O_q\overset{\eqref{Oq},\eqref{qy}}{:=}\clos O_q$ имеем 
\begin{multline}\label{rad0}
l_{\eta}^{\rght}(r,R)
\overset{\eqref{df:dDlm+},\eqref{eta1},\eqref{Oq}}{=} 
\int_{\substack{r<|z|\leq R,\\z\in \overline O_q, \Re z >0}}\Re \frac{1}{ z} \dd \eta (z)
\\
\overset{\eqref{eta1},\eqref{Oq}}{\leq}  
\int_{\substack{r<|z|\leq R,\\z\in \overline O_q,\Re z >0}}\frac{k(|\Im z|)|\Im z|}{|z|^2} \dd \eta(z)
\\
\leq\int_{\substack{r<|z|\leq R,\\ z\in \overline O_q}}\frac{k(|\Im z|)|z|}{|z|^2} \dd \eta(z)
\leq \sup_{\stackrel{|z|\geq r,}{z\in \overline O_q}}k(|\Im z|) \int_r^R\frac{1}{t} \dd \eta^{\rad}(t)
\\
\overset{\eqref{eta1}}{\leq}  \const \sup_{\stackrel{|z|\geq r}{{z\in \overline  O_q}}}k(|\Im z|) \log \frac{R}{r}+\const,
\end{multline}
где постоянная $\const$ не зависит от $R, r\geq 1$. По построению множества $O_q$ 
в \eqref{Oq} при выборе функций $k$ и $q$ как в  \eqref{qy}--\eqref{k0} легко видеть, что 
\begin{equation}\label{supk}
\sup\bigl\{ k(|\Im z|) \colon |z|\geq r, z\in \overline  O_q \bigr\}=o(1), \quad r\to +\infty.
\end{equation}
Таким образом, из \eqref{k}, \eqref{rad0} и \eqref{supk} получаем \eqref{eta0}.
\end{proof}
Применение леммы \ref{lab} к мере $\eta=\mu_M^q$ из \eqref{muMq} даёт
согласно \eqref{lUMqm} требуемое соотношение \eqref{lZZ+0} с функцией $m$
вида  \eqref{m0}.

\subsubsection{\bf Завершение {\bf c}\ref{c3}$\overset{\mathbf{?}}{\Rightarrow}${\bf c}\ref{c2}}
Пусть $\e\in \RR_*^+$, $q(iy)=\e |y|$, $y\in \RR$. Используем обозначения
\eqref{mud}, оценку \eqref{lUMqm} и существование конечного числа 
\begin{equation}\label{CM}
m=\sup_{t\geq 1} \frac{\mu_M^{\rad}(t)}{t},
\end{equation}
не зависящего от $\e$. Имеют место оценки 
\begin{multline*}
l_{\mu_M^q}(r,R)+J(r,R;q)\leq \e\int_{\stackrel{r<|z|\leq R,}{|\Re z|\leq \e|\Im z|}}
\Re\frac{1}{z} \dd \mu_M(z) +\e\log\frac{R}{r}\\
\leq \int_r^R\frac{\e}{t}\dd \mu_M^{\rad}(t) +\e\log\frac{R}{r}
\overset{\eqref{CM}}{\leq} \e m+
\e \int_r^R \frac{\mu_M(t)}{t^2}\dd t+ \e\log\frac{R}{r}\\
\leq \e m+\e m\log\frac{R}{r}+\e \log\frac{R}{r}
=\e (m+1)\log\frac{R}{r}+\e m, \quad
1\leq r<R<+\infty.
\end{multline*}
 Отсюда в силу произвола в выборе $\e\in \RR_*^+$ из \eqref{lUMqm} получаем 
требуемое \eqref{lZZ+} для любого $\e\in \RR_*^+$. Импликация 
{\bf c}\ref{c3}$\Rightarrow${\bf c}\ref{c2} доказана. 
\begin{remark}\label{rem3_2}
При доказательстве  импликаций  
{\bf x}\ref{a3}${\Rightarrow}${\bf x}\ref{a2} использованы лишь ограничения \eqref{suppm} на меры $\nu,\mu\in \Meas^+(\CC)$, а именно: 
\begin{enumerate}[1)]
\item\label{1_c-} меры $\nu,\mu$ \textit{конечного типа при порядке\/} $1$, что  \textit{необходимо\/}  для существования  $U, M\in \sbh_*(\CC)$ с мерами Рисса  соотв. $\upsilon_U\geq \nu$, $\mu_M\geq \mu$; 
\item\label{2_c-}  $\supp \mu\overset{\eqref{suppm}}{\subset} \CC_{\overline \rght}$, что фактически тоже \textit{необходимо\/} для возможности перехода от 
 функции интервалов $l_{\mu_M}$ к 
$l_{\mu}$ при выводе {\bf x}\ref{a2}.
\end{enumerate}
Представления же {\bf [meas]}\eqref{numu} для мер $\nu,\mu\in \Meas^+(\CC)$
не применялись.  
\end{remark}

\section{Двустороннее выметание на мнимую ось $i\RR$}\label{bal_iR}
\setcounter{equation}{0} 
Сформулированные в этом подразделе факты 
--- сводка основных  результатов второй части \cite{KhII} нашей работы
применительно  к выметанию рода $1$ на систему из двух лучей
$\{i\RR^+, -i\RR^+ \}$. Такое выметание представляет собой один из ключевых этапов доказательства теоремы \ref{th1}  --- основы доказательства теорем  \ref{th2}--\ref{th3}. При этом система двух лучей $\{i\RR^+, -i\RR^+ \}$ как точечное множество отождествляется с мнимой осью $i\RR$.  В частности, когда речь идет о выметании на пару лучей  $\{i\RR^+, -i\RR^+ \}$, то говорим о выметании на мнимую ось $i\RR$. Процедура выметания рода $1$ на мнимую ось $i\RR$ заключается в сочетании двух выметаний рода $1$:  отдельно из правой полуплоскости $\CC_{\rght}$  и из левой полуплоскости $_{\lft}\CC$.  Таким образом, значительная часть результатов настоящего параграфа о двустороннем выметании следует из результатов о выметании рода $q=1$ из верхней полуплоскости $\CC^{\up}$ на $\RR$, изложенных в \cite[3.2, 4.1]{KhII}, а также в  \cite[\S~3]{kh91AA}, \cite[гл.~II]{KhDD92}.

\subsection{Двустороннее выметание на $i\RR$ заряда}
 {\it Двусторонний гармонический заряд рода $1$ для\/ $\CC\setminus i\RR$ в точке\/ $z\in \CC$} определяем на интервалах $[iy_1,iy_2]$, $-\infty<y_1<y_2<+\infty$, как функцию интервалов  \cite[определение 3.1]{kh91AA}, \cite[определение 2.1.1]{KhDD92},
\cite[определение 2.1]{KhII} 
\begin{multline}\label{se:HC+1}
\Omega\bigl(z,[iy_1,iy_2]\bigr)\overset{\text{\cite[(2.1)]{KhII}}}{:=}
\Omega_{\CC\setminus i\RR}^{[1]}\bigl(z,[iy_1,iy_2]\bigr)\\
\overset{\eqref{omega}}{:=}\omega\bigl(z,[iy_1, iy_2]\,\bigr)
-\frac{y_2-y_1}{\pi}  \Bigl|\Re \frac{1}{z}\Bigr| \, , \quad z\neq 0.
\end{multline}
Пусть $\nu\in \Meas(\CC)$ --- заряд конечного типа при порядке $p$ с целой частью   $[p]\leq 1$. Определим, следуя  \cite[определение 3.2]{kh91AA}, \cite[определение 2.1.2]{KhDD92},  \cite[определение 3.1, теорема 1, замечание 3.3]{KhII}, {\it двустороннее выметание\/ $\nu^{\Bal}\in \Meas (i\RR)$ заряда $\nu$ из $\CC\setminus i\RR$ на $i\RR$,\/} которое  через функцию распределения \eqref{nuiR} локально ограниченной вариации на $i\RR$ можно задать как 
 \begin{subequations}\label{df:baliR}
\begin{align}
(\nu^{\Bal})^{i\RR}(iy)
&\overset{\eqref{se:HC+1}}{:=}\int_{(\CC\setminus i\RR)\cap D(r_0)} \omega\bigl(z, [0,iy]\bigr)\dd \nu(z)
\notag \\
&+\int_{(\CC\setminus i\RR)\setminus D(r_0)} \Omega \bigl(z, [0,iy]\bigr)\dd \nu(z)+ \nu ([0,iy])\quad\text{при $y\geq 0$},
\tag{\ref{df:baliR}+}\label{df:baliR+}
\\
(\nu^{\Bal})^{i\RR}(iy)&\overset{\eqref{se:HC+1}}{:=}
-\int_{(\CC\setminus i\RR)\cap D(r_0)} \omega\bigl(z, [iy,0)\bigr)\dd \nu(z)
\notag \\
&-\int_{(\CC\setminus i\RR)\setminus D(r_0)}  \Omega\bigl(z, [iy,0)\bigr)\dd \nu(z)- \nu ([iy,0))\quad\text{при $y<0$}.
\tag{\ref{df:baliR}-}\label{df:baliR-}
\end{align}
\end{subequations} 

\begin{theoB}\label{theoB1} Пусть $\nu\in \Meas(\CC)$ и  ${\type_1}[\nu] \overset{\eqref{fden}}{<}+\infty$. Тогда существует двустороннее выметание $\nu^{\Bal}\overset{\eqref{df:baliR}}{\in} \Meas(i\RR)$, удовлетворяющее условиям\/  \cite[лемма 2.1.2, (1.16)]{KhDD92}, \cite[теорема 3, п.~2]{KhII}
\begin{subequations}\label{iRbnu}
\begin{align}
\bigl|\nu^{\Bal}\bigr|^{\rad}(r)&=O(r\log r), \quad r\to +\infty,
\tag{\ref{iRbnu}$\infty$}\label{iRbnui}
\\
\intertext{и при $0\notin \supp \nu$ \cite[лемма 2.1.2, (1.17)]{KhDD92}, \cite[теорема 1]{KhII}}
\bigl|\nu^{\Bal}\bigr|^{\rad}(r)&=O(r^2), \quad r\to 0.
\tag{\ref{iRbnu}o}\label{iRbnu0}
\end{align}
\end{subequations} 
Если заряд $\nu$ удовлетворяет двустороннему условию Бляшке \eqref{cB2},
то ${\type_1}[\nu^{\Bal}] \overset{\eqref{fden}}{<}+\infty$ \cite[теорема 3.1]{kh91AA}, \cite[теорема 2.2.1]{KhDD92}, \cite[теорема 3, п.~4]{KhII}, а также
\begin{enumerate}[{\rm (i)}]
\item\label{cL} при дополнительном условии Линделёфа рода $1$ из \eqref{con:Lp+}
как выметание $\nu^{\Bal}\in \Meas(i\RR)$, так и разность зарядов $\nu-\nu^{\Bal}\in \Meas(\CC)$ удовлетворяют условию Линделёфа 
\cite[теорема 3.2]{kh91AA}, \cite[теорема 2.3.1]{KhDD92};

\item\label{aei0} если  заряд $\nu$ еще и отделён от мнимой оси  $i\RR$
в смысле \eqref{con:dism}, 
то найдутся   числа $C\in \RR^+$ и $y_0\in \RR_*^+$, с которыми 
\begin{equation}\label{trnuair}
\bigl|\nu^{\Bal}\bigr|^{i\RR}(iy+ir)-\bigl|\nu^{\Bal}\bigr|^{i\RR}(iy-ir)\leq Cr\; \text{при всех $|y|\geq y_0$, $r\in [0,1]$}
\end{equation}
{\rm (в неявной форме см. \cite[теорема 3.3]{kh91AA}, в явном виде  --- \cite[теорема 2.2.2]{KhDD92}) и \cite[следствие 3.1, п.~(ii), (3.24)]{KhII}).}
\end{enumerate}
\end{theoB}

\subsection{Двустороннее выметание рода\/ $1$ $\delta$-субгармонической функции на мнимую ось\/ $i\RR$}
Пусть $v\in \dsbh_*(\CC)$. Функцию $v^{{\Bal}}\in \dsbh_*(\CC)$ называем {\it выметанием функции\/ $v$ из\/ $\CC\setminus i\RR$ на $i\RR$}, если $v^{{\Bal}}=v$ на\/ $i\RR$ вне полярного множества  и сужение  $v^{{\Bal}}\bigm|_{\CC\setminus i\RR}$ --- гармоническая функция на   $\CC\setminus i\RR$ \cite[определение 4.1]{KhII}.

\begin{theoB}[{\rm \cite[теорема 2.1.1]{KhDD92}}]\label{Balv} Пусть $v\in \dsbh_*(\CC)$ c зарядом Рисса $\nu$ \underline{конечного типа} при порядке $1$ и функция $v$ представима в виде разности 
\begin{equation}\label{reprv}
v:=v_+-v_-, \quad v_{\pm}\sbh_*(\CC), \quad \ord_{\infty}[v_{\pm}]\overset{\eqref{ord}}{\leq} 1.
\end{equation} 
 Тогда \underline{существует выметание\/} $v^{\Bal}\in \dsbh_*(\CC)$ из\/ $\CC\setminus i\RR$ на $i\RR$ c зарядом Рисса $\nu^{\Bal}$, представимое в виде разности 
\begin{equation}\label{reprvB}
v^{\Bal}:=u_+-u_-, \quad u_{\pm}\in \sbh_*(\CC), \quad \ord_{\infty}[u_{\pm}]\overset{\eqref{ord}}{\leq}  1.
\end{equation} 
В дополнение допустим, что в представлении  \eqref{reprv} функции $v_{\pm}$ конечного типа при порядке $1$, т.\,е.  ${\type_1}[v_{\pm}]\overset{\eqref{typev}}{<}+\infty$, а также заряд Рисса $\nu$
удовлетворяет слабому двустороннему условию Бляшке \eqref{cB2} и условию Линделёфа рода $1$ из \eqref{con:Lp+}. Тогда две функции $u_{\pm}$ в \eqref{reprvB} можно выбрать так, что $\type_1[u_{\pm}]<+\infty$, т.\,е. функция $v^{\Bal}\in \dsbh_*(\CC)$ конечного типа, а также, если  заряд $\nu$ еще и \underline{отделён от мнимой оси}  $i\RR$, для некоторых $C\in \RR^+$ и $y_0\in \RR_*^+$ имеем неравенства 
\begin{equation}\label{trnuair_u}
u_\pm (iy)\overset{\eqref{df:MCBc}}{\leq} C_{u_\pm} (iy,1) \leq u_\pm (iy)+ C\quad \text{при всех $|y|\geq y_0$}.
\end{equation}
\end{theoB}

\begin{remark}[{\it Заключительное к настоящей третьей части}]
Оставшиеся недоказанными три импликации 
Полные доказательства импликаций {\bf x}\ref{a3}${\Rightarrow}${\bf x}\ref{a1}  при 
$\bf x:=a,b,c$  для основных теорем \ref{th1}--\ref{th3}  на основе результатов из  
\S~\ref{bal_iR} и их развития будут даны в четвёртой части нашего исследования, оформление которой близк\'о к завершению. 

\end{remark}

\end{document}